\documentclass[reqno,11pt]{amsart}

\usepackage{graphics}
\usepackage{amssymb, amsmath}
\usepackage{version}
\usepackage{fancybox}
\usepackage{bm}
\usepackage{xcolor}
\usepackage{mathrsfs}
\usepackage{mathtools} 
\usepackage{esint}
\theoremstyle{theorem}
\newtheorem{theorem}{\sc Theorem}[section]
\newtheorem{lemma}[theorem]{\sc Lemma}

\newtheorem{proposition}[theorem]{\sc Proposition}
\newtheorem{corollary}[theorem]{\sc Corollary}

\theoremstyle{definition}
\newtheorem{definition}[theorem]{\sc Definition}
\newtheorem*{notation*}{Notation}

\theoremstyle{remark}
\newtheorem{remark}[theorem]{\sc Remark}

\makeatletter

\@addtoreset{equation}{section}

\renewcommand{\d}{\text{\rm d}}
\newcommand{\D}{\text{\rm D}}
\newcommand{\vep}{\varepsilon}

\newcommand{\R}{\mathbb{R}}
\renewcommand{\P}{{\sf (P)}}
\newcommand{\Pk}{${\sf (P)}_k$}
\newcommand{\BC}{{\sf (BC)}}

\newcommand{\inz}{\in\mathbb{Z}}
\newcommand{\inr}{\in\mathbb{R}}
\newcommand{\dv}{{\rm{div}}}
\newcommand{\loc}{\mathrm{loc}}
\newcommand{\supp}{\mathrm{supp}\,}
\newcommand{\diam}{\mathrm{diam}\,}

\allowdisplaybreaks[4]

\newcommand{\prf}[1]{}  
\begin{document}
\title[$p$-Laplace elliptic systems for locally integrable forcing]{Existence of distributional solutions to degenerate elliptic systems for locally integrable forcing\prf{\\{\tt [Extended version]}}}
%
%
\author{Goro Akagi}
\address[Goro Akagi]{Mathematical Institute and Graduate School of Science, Tohoku University, Aoba, Sendai 980-8578, Japan}
\email{goro.akagi@tohoku.ac.jp}
\author{Hiroki Miyakawa}
\address[Hiroki Miyakawa]{Aioi Nissay Dowa Insurance Co., Ltd., 1-28-1 Ebisu Shibuya Tokyo 150-8488, Japan}
\email{hiroki.miyakawa.t1@alumni.tohoku.ac.jp}
%
%
\date{\today}
\dedicatory{Dedicated to the memory of Professor Marek Fila}
\maketitle
\begin{abstract}
This paper presents an existence result and maximal regularity estimates for distributional solutions to degenerate/singular elliptic systems of $p$-Laplacian type with absorption and (prescribed) locally integrable forcing posed in (possibly unbounded) Lipschitz domains. In particular, the forcing terms may not belong to the dual space of an energy space, e.g., $W^{1,p}_{\rm loc}$, which is necessary for the existence of weak (or energy) solutions of class $W^{1,p}_{\rm loc}$. The method of a proof relies on both local energy estimates and a relative truncation technique developed by Bul\'{i}\v{c}ek and Schwarzacher (Calc.~Var.~PDEs in 2016), where the bounded domain case is studied for (globally) integrable forcing.
\end{abstract}

\section{Introduction}

Let $d, N \inz_{\geq 1}$ and let $\Omega$ be a (possibly unbounded) Lipschitz domain of $\mathbb{R}^d$. Let $1 < p < r < \infty$ and let $q,s \in (1, \infty)$ satisfy
\begin{equation}\label{exp}
p-1 \leq q < p \quad \mbox{ and } \quad r-1 \leq s < r.
\end{equation}
In the present paper, we consider the Dirichlet problem for the following degenerate/singular elliptic system of $p$-Laplacian type with absorption and forcing terms:
$$
\left\{
\begin{alignedat}{4}
& - \dv \, A(\,\cdot\,,\nabla u) + |u|^{r-2}u = - \,\dv \, (|f|^{p-2}f)+|g|^{r-2}g \ &&\mbox{ in } \
\Omega,\\
& u = 0 \ &&\mbox{ on } \ \partial \Omega,
\end{alignedat}
\right.\leqno{\P}
$$
where $u : \Omega \to \mathbb{R}^N$ is an unknown function, and $f \in L^q_{\loc}(\overline{\Omega};\mathbb{R}^{d\times N})$ and $g\in L^s_{{\loc}}(\overline{\Omega};\mathbb{R}^N)$ are prescribed. In particular, when $\Omega = \R^d$, we ignore the Dirichlet boundary condition. Moreover, $A : \Omega \times \mathbb{R}^{d\times N} \to \mathbb{R}^{d\times N}$ is a Carath\'eodory function complying with the following assumptions: There exist constants $C_1, C_2 > 0$ and nonnegative functions $\beta_1 \in L^{1}_{\loc}(\overline{\Omega})$, $\beta_2 \in L^{p'}_{\loc}(\overline{\Omega})$ such that
\begin{gather}
A(x,z) : z \geq C_1|z|^p - \beta_1(x),\label{ineq:ellipAcoer}\\
|A(x,z)| \leq C_2|z|^{p-1} + \beta_2(x),\label{ineq:ellipAbdd}\\
(A(x,z_1)-A(x,z_2)) : (z_1-z_2) \geq 0\label{ineq:ellipAmono}
\end{gather}
for $z,z_1,z_2 \inr^{d\times N}$ and a.e.~$x \in \Omega$. Hence a typical example of the elliptic operators satisfying \eqref{ineq:ellipAcoer}--\eqref{ineq:ellipAmono} is the so-called $p$-Laplace operator,
$$
-\dv \,A(\,\cdot\,,\nabla u) = - \Delta_p u := - \,\dv\, (|\nabla u|^{p-2}\nabla u),
$$
for which we set $A(x,z) = |z|^{p-2}z$ for $z \inr^{d\times N}$ and a.e.~$x \in \Omega$ obviously.

In case $f \in L^p(\Omega;\mathbb{R}^{d\times N})$ and $g \in L^r(\Omega;\mathbb{R}^N)$ (i.e., they are globally integrable over $\Omega$ with $p = q$ and $r = s$), existence of weak solutions can be proved in a standard manner (see Proposition \ref{prop:approxsol} below). On the other hand, in case either $p > q$ or $r > s$ holds, \P \ admits no weak solution for general $f \in L^q(\Omega;\mathbb{R}^{d\times N})$ and $g \in L^s(\Omega;\mathbb{R}^N)$ due to a simple comparison of both sides of the equation. However, we can still expect existence of \emph{distributional} solutions to \P \ (see Definition \ref{D:sol} below). The main purpose of the present paper is to prove existence of such distributional solutions to \P \ for $f \in L^q_{\loc}(\overline{\Omega};\mathbb{R}^{d\times N})$ and $g\in L^s_{{\loc}}(\overline{\Omega};\mathbb{R}^N)$ with exponents $q,s \in (1,\infty)$ satisfying \eqref{exp} (whose lower bounds are necessary for the existence of distributional solutions) as well as to establish \emph{maximal regularity estimates} for such distributional solutions; this is a reason why we consider the forcing term of this form in \P. To be more precise, we expect that $\nabla u$ and $u$ can be estimated in terms of $f$ and $g$ in $L^q_{\rm loc}$ and $L^s_{\rm loc}$, respectively.

Elliptic PDEs have been vigorously studied in a vast amount of literatures (see, e.g.,~\cite{GT}). Among those, we here start with existence and regularity of locally integrable solutions to elliptic PDEs with locally integrable forcing. As is well known, an appropriate growth condition at infinity is imposed on the forcing term $F$ in order to assure existence of solutions to the Poisson equation in $\R^d$,
$$
-\Delta u = F \ \mbox{ in } \mathbb{R}^d.
$$
In~\cite{Bre84}, existence and uniqueness of locally integrable (distributional) solutions $u \in L^{r-1}_{\rm loc}(\mathbb{R}^d)$ are verified for the following elliptic PDE with an absorption for $2 < r < \infty$:
\begin{equation}\label{Br}
-\Delta u + |u|^{r-2}u = F \ \mbox{ in } \mathbb{R}^d
\end{equation}
without prescribing behaviors of $F \in L^1_{\rm loc}(\mathbb{R}^d)$ at infinity, and moreover, it holds that
$$
\int_{B_R} |u|^{r-1} \, \d x \leq C\left(1 + \int_{B_{R'}} |F| \, \d x\right)
$$
for any $R' > R > 0$. Such an elliptic equation \eqref{Br} with an absorption naturally appears as a stationary problem for the heat equation with an absorption term. Moreover,~\cite{Bre84} may also be inspired by studies on existence of solutions to parabolic equations with growing initial data (see, e.g.,~\cite{Evans}). Actually, in~\cite{Bre84}, a parabolic counterpart of the result for \eqref{Br} is also briefly discussed (e.g., solvability of the heat equation with an absorption is considered for any locally integrable initial data). In~\cite{BGV93}, the aforementioned result is also extended to degenerate and singular elliptic PDEs of $p$-Laplacian type such as
\begin{equation}\label{BGV}
-\Delta_p u + |u|^{r-2}u = F \ \mbox{ in } \mathbb{R}^d
\end{equation}
for $F \in L^1_{\rm loc}(\R^d)$ and $r > p$. The above case is much more technically involved than the semilinear case. Indeed, in contrast with~\cite{Bre84}, one has to establish a gradient estimate for certain approximate solutions $(u_n)$ due to the nonlinearity arising from the $p$-Laplacian. It is actually done in~\cite[Lemma 2.2]{BGV93} but in a totally different way from the present paper. Moreover, in order to identify a weak limit of the $p$-Laplacian term $\Delta_p u_n$, a pointwise convergence of the gradients $\nabla u_n$ is also derived by use of a strong monotonicity of the $p$-Laplacian (see assumption (A4) of~\cite{BGV93}), which is no longer assumed in the present paper (cf.~see \eqref{ineq:ellipAmono}).

On the other hand, the optimal regularity of the (distributional) solution is not fully pursued in these literatures. Actually, the Calder\'on--Zygmund estimate for the Laplacian fails in $L^1$. Moreover, it is rather delicate to determine the optimal regularity of solutions to degenerate and singular nonlinear elliptic PDEs such as \eqref{BGV} even posed in bounded domains.

A regularity theory for degenerate and singular elliptic PDEs, often referred as a \emph{nonlinear Calder\'on--Zygmund theory}, has been remarkably developed in the last few decades. It is concerned with weak solutions $u \in W^{1,p}_0(\Omega)$ to the (homogeneous) Dirichlet problem of the form,
\begin{equation}\label{pLap-div}
- \,\dv \, \left(|\nabla u|^{p-2}\nabla u \right) = - \, \dv \, \left(|f|^{p-2}f \right) \ \mbox{ in } \Omega,
\end{equation}
where $\Omega$ is a bounded domain of $\R^d$ and $1 < p < \infty$, and provides a maximal regularity estimate,
\begin{equation}\label{nCZ}
\|\nabla u\|_{L^q(\Omega)} \leq C \|f\|_{L^q(\Omega)}
\end{equation}
for $f \in L^q(\Omega)$ and $q \geq p$. See~\cite{Iwa83},~\cite[Theorem 4.1]{GiaGiu82},~\cite{AcMi05},~\cite[Theorem 7.8]{KrMi06}. Here we also remark that the estimate \eqref{nCZ} can fail for vectorial cases with $q > p = 2$ large enough (see~\cite{Nec77}). We further refer the reader to survey papers~\cite{Min10,Min11,Min17} and references therein. On the other hand, its extension to the case where $1 < q < p$ is not straightforward at all in contrast with the linear case, for which a duality argument works. Moreover, even existence of (distributional) solutions is already highly nontrivial when $q < p$. In~\cite{Iwa92} and~\cite{Lew93}, maximal regularity estimates are established with $q < p$ close enough to $p$ for distributional solutions defined on $\R^d$, provided that they exist. Moreover, an existence result and maximal regularity estimates for distributional solutions to \eqref{pLap-div} are first proved for bounded domains and $q < p$ close enough to $p$ by Bul\'{i}\v{c}ek and Schwarzacher in~\cite{BuSch16}, which relies on sophisticated techniques such as an \emph{$\mathcal{A}_p$-weighted biting div-curl lemma} developed in~\cite[Theorem 2.6]{BDS} as well as a \emph{relative truncation technique} for establishing maximal regularity estimates. 

The aim of the present paper may also be regarded as an extension of the nonlinear Calder\'on--Zygmund theory to the elliptic system \P \ with an arbitrary locally integrable forcing $f \in L^q_{\rm loc}(\overline{\Omega};\R^{d\times N})$ and $g \in L^s_{\rm loc}(\overline{\Omega};\R^N)$ for $q$ and $s$ being in the range \eqref{exp}. Here we emphasize again that the presence of the absorption term in \P \ is necessary to prove an existence result as mentioned above (see~\cite{Bre84,BGV93}), and furthermore, the forcing terms of the equation in \P \ are designed in such a way as to discuss the maximal regularity of gradients of distributional solutions.

Finally, we give a couple of related references without any claim of the completeness, although all these results are devoted to equations of different type posed in \emph{bounded} domains only. The results in~\cite{BDS,BuSch16} have been extended to stationary Navier--Stokes equations with non-Newtonian stress tensors in~\cite{MiSc21}, where a solenoidal Lipschitz truncation is also developed (cf.~see also~\cite{BBS16}, which is a generalization of~\cite{BDS}). One can find parabolic counterparts of these studies in, e.g.,~\cite{BBS19,BBKOS} and references therein.

The present paper consists of six sections. In Section \ref{S:result}, we state a main result, which is concerned with maximal regularity estimates for distributional solutions to \P. Section \ref{S:prelim} is devoted to recalling preliminary facts on the Muckenhoupt class $\mathcal{A}_p$, an $\mathcal{A}_p$-weighted biting div-curl lemma and the Whitney decomposition of open sets for later use. In Section \ref{S:weak}, we assure existence and uniqueness of weak solutions to \P \ for $f \in L^p(\Omega;\R^{d\times N})$ and $g \in L^r(\Omega;\R^N)$. Section \ref{S:loc_en_est} provides weighted local energy estimates for weak solutions to \P, which play a crucial role to prove the main result and where main technical novelty may reside. Finally, in Section \ref{S:proof}, we complete a proof of the main result.

\begin{notation*}
For $a,b \inr^d$, $c,d \inr^N$ and $A,B \inr^{d\times N}$, we denote by $a \cdot b$, $c\cdot d$ and $A : B$ (standard) inner products, that is,
$$
a \cdot b = \sum_{j=1}^d a_jb_j, \quad c \cdot d = \sum_{k=1}^N c_k d_k, \quad A : B = \sum_{j=1}^d\sum_{k=1}^N A_{jk}B_{jk},
$$
respectively. Then $|a|$, $|c|$ and $|A|$ denote standard Euclidean norms of $a$, $c$ and $A$, respectively, that is, 
$$
|a| = \sqrt{a \cdot a}, \quad |c| = \sqrt{c \cdot c}, \quad |A| = \sqrt{A : A}.
$$
Moreover, $a \otimes c \in \R^{d \times N}$ is a direct product of $a$ and $c$, that is, $(a \otimes c)_{ik} = a_i c_k$. Hence we have $|a \otimes c| = |a| |c|$. For $u = (u_k): \Omega \to \R^N$ and $F = (F_{jk}) : \Omega \to \R^{\d \times N}$ smooth enough, $\nabla u : \Omega \to \R^{d \times N}$ and $\dv \, F : \Omega \to \R^N$ are defined as
$$
(\nabla u)_{jk} = \partial_{x_j} u_k = \frac{\partial u_k}{\partial x_j}, \quad (\dv \, F)_k = \sum_{i=1}^d \partial_{x_i} F_{ik}
$$
for $j = 1,\ldots,d$ and $k = 1,\ldots,N$. For $x \inr^d$ and $R > 0$, we write $B_R(x) := \{ y \inr^d \colon |x-y| < R \}$ and $B_R := B_R(0)$. Furthermore, for any domain $\Omega$ of $\R^d$, we simply write $\Omega_R := \Omega \cap B_R$. For each $f \in L^1_{\loc}(\mathbb{R}^d)$, we denote by $Mf$ the Hardy--Littlewood maximal function, that is,
$$
(Mf)(x) = \sup_{r>0} \, \fint_{B_r(x)} |{f}(y)| \,\d y
$$
for $x \in \mathbb{R}^d$. Here and henceforth, we write 
$$
\fint_B g(y)\,\d y = |B|^{-1}\int_B g(y) \,\d y
$$
for any $g \in L^1_{\rm loc}(\R^d)$ and open ball $B$ in $\R^d$. For $p \in [1,\infty]$, we denote by $p'$ the H\"{o}lder conjugate of $p$, that is, $p' = p/(p-1)$, $1' = \infty$ and $\infty' = 1$. Moreover, $C$ stands for a generic nonnegative constant which may vary from line to line.
\end{notation*}

\section{Main Result}\label{S:result}

In order to impose the (homogeneous) Dirichlet boundary condition on locally integrable distributional solutions to the Dirichlet problem \P, let us first set up its weak formulation below. Let $B_R$ denote the open ball in $\R^d$ centered at the origin of radius $R$ and let $(\rho_R)$ be a family of smooth cut-off functions satisfying
\begin{equation}\label{cut-off}
\left\{
\begin{aligned}
&\rho_R \in C^\infty_c(\R^d), \quad 0 \leq \rho_R \leq 1 \ \mbox{ in } \mathbb{R}^d,\\
&\rho_R \equiv 1 \ \mbox{ on } B_R, \quad \supp \rho_R \subset \overline{B_{2R}}
\end{aligned}
\right. 
\end{equation}
for any $R > 0$. In what follows, the homogeneous Dirichlet boundary condition for $u \in W^{1,1}_{\rm loc}(\overline{\Omega};\mathbb{R}^N)$ will be regarded (in a weak sense) as
$$
u \rho_R \in W_0^{1,1}(\Omega;\mathbb{R}^N) \ \mbox{ for any } R > 0.
\leqno{\sf (BC)}
$$
Indeed, if $u \in C(\overline{\Omega};\mathbb{R}^N)$, then one finds that $\mathrm{Tr}(u\rho_R) = (u\rho_R)|_{\partial \Omega}$ for any $R > 0$, whence it follows from \BC \ that $u\rho_R = 0$ on $\partial\Omega$ for any $R>0$. Here $\mathrm{Tr}:W^{1,1}(\Omega;\mathbb{R}^N)\to L^{1}(\partial\Omega;\mathbb{R}^N)$ denotes the trace operator. Therefore from the arbitrariness of $R > 0$, the homogeneous Dirichlet condition, i.e., $u=0$ on $\partial\Omega$, follows.

We now set up notions of weak and distributional solutions to \P.
\begin{definition}\label{D:sol}
Let $1 < p < r < \infty$ and let $f \in L^{\max\{1,p-1\}}_{\rm loc}(\overline{\Omega};\R^{d\times N})$ and $g \in L^{\max\{1,r-1\}}_{\rm loc}(\overline{\Omega};\R^N)$.
\begin{enumerate}
\item Set $X^{p,r}(\Omega;\mathbb{R}^N) := \overline{C_c^\infty(\Omega;\mathbb{R}^N)}^{\|\cdot\|_{X^{p,r}}}$ equipped with the norm,
$$
\|u\|_{X^{p,r}}:=\|\nabla u\|_{L^p(\Omega;\mathbb{R}^{d\times N})}+\|u\|_{L^r(\Omega;\mathbb{R}^N)} 
$$
for $u \in X^{p,r}(\Omega;\mathbb{R}^N)$.
\item A function $u \in X^{p,r}(\Omega;\mathbb{R}^N)$ is called a \emph{weak solution} to \P, if it holds that
\begin{align*}
\lefteqn{
\int_\Omega A(\cdot,\nabla u) : \nabla \varphi \,\d x + \int_\Omega |u|^{r-2}u \cdot \varphi \,\d x
}\\
& = \int_\Omega |f|^{p-2}f : \nabla \varphi \,\d x + \int_\Omega |g|^{r-2}g \cdot \varphi \,\d x
\end{align*}
for any $\varphi \in C_c^\infty(\Omega;\mathbb{R}^N)$.
\item A function $u \in W_{\loc}^{1,\max\{1,p-1\}}(\overline{\Omega};\mathbb{R}^N) \cap L_{\loc}^{\max\{1,r-1\}}(\overline{\Omega};\mathbb{R}^N)$ is called a \emph{distributional solution} to \P, if it holds that
\begin{align*}
\lefteqn{
\int_\Omega A(\cdot,\nabla u) : \nabla\varphi \,\d x + \int_\Omega |u|^{r-2}u \cdot \varphi \,\d x
}\\
& = \int_\Omega |f|^{p-2}f : \nabla \varphi \,\d x + \int_\Omega |g|^{r-2}g \cdot \varphi \,\d x
\end{align*}
for any $\varphi \in C_c^\infty(\Omega;\mathbb{R}^N)$, and moreover, \BC \ holds for an arbitrary family $(\rho_R)$ of cut-off functions satisfying \eqref{cut-off}.
\end{enumerate}
\end{definition}

Our main result reads,

\begin{theorem}\label{thm:elliplocDS}
Let $\Omega$ be a {\rm (}possibly unbounded{\rm )} Lipschitz domain of $\mathbb{R}^d$ and let $1 < p < r < \infty$. Then there exists a positive constant $\varepsilon_0 = \varepsilon_0(C_1,C_2,\Omega,d,N,p,r) \in (0,p/r]$ satisfying the following\/{\rm :} Let $\varepsilon \in (0,\varepsilon_0)$ be fixed and set
$$
q := p - \varepsilon \in (p-1,p) \quad \mbox{ and } \quad s := \frac{p-\vep }p r \in (r-1,r).
$$
Then for any $f \in L_{\loc}^q(\overline{\Omega};\mathbb{R}^{d\times N})$ and $g \in L_{\loc}^s(\overline{\Omega};\mathbb{R}^{N})$, the Dirichlet problem \P \ admits a distributional solution $u \in W^{1,q}_{\loc}(\overline{\Omega};\mathbb{R}^N) \cap L^{s}_{\loc}(\overline{\Omega};\mathbb{R}^N)$ such that
\begin{align}
\MoveEqLeft{
\int_{\Omega_{R}} |\nabla u|^p \omega \,\d x + \int_{\Omega_{R}} |u|^r \omega \,\d x
}\notag\\
&\leq C_0 \left( \int_{\Omega_{2R}} |f|^p \omega \,\d x + \int_{\Omega_{2R}} |g|^r \omega \,\d x \right. \notag\\
&\qquad + \left. \int_{\Omega_{2R}} \beta_1 \omega\,\d x + \int_{\Omega_{2R}}\beta_2^{p'} \omega\,\d x + \delta^{-\vep} R^{d - \frac{pr}{r-p}} \right) \label{thm:Localest}
\end{align}
for any $R > 0$ and $\delta \in (0,1]$. Here $C_0 = C_0(C_1,C_2,\Omega,d,N,p,r,\varepsilon)$ is a positive constant and a function $\omega : \R^d \to (0,\infty)$ is given by
$$
\omega(x) = \left( M \left[ \big( |\bar f|+|\bar g|^{s/q} \big) \chi_{\Omega_{2R}} \right](x) + \delta \right)^{-\vep} \ \mbox{ for } \ x \in \R^d, 
$$
where $\bar f, \bar g$ denote the zero extensions onto $\R^d$ of $f,g$, respectively, and $\chi_{\Omega_{2R}}$ stands for the characteristic function supported over $\Omega_{2R}$.
\end{theorem}

Moreover, we have

\begin{corollary}\label{C:main}
Under the same assumptions as in Theorem {\rm \ref{thm:elliplocDS}} with $\beta_1 = \beta_2 = 0$, there exists a positive constant $\tilde{C}_0 = \tilde{C}_0(C_1,C_2,\Omega,d,N,p,r,\varepsilon)$ such that
\begin{align}
\MoveEqLeft{
\int_{\Omega_{R}} |\nabla u|^q \,\d x + \int_{\Omega_{R}} |u|^s \,\d x
}\notag\\
&\leq \tilde{C}_0 \left( \int_{\Omega_{2R}} |f|^q \,\d x + \int_{\Omega_{2R}} |g|^s \,\d x 
+ R^{d - \frac 12 \frac{pr}{r-p}} 
\right)\label{thm:Localest2}
\end{align}
for any $R > 0$.
\end{corollary}

\section{Preliminaries}\label{S:prelim}

In this section, we shall set up preliminary facts to be used.

\subsection{Muckenhoupt weights}

We first recall the class $\mathcal{A}_p$ of Muckenhoupt weights $\omega \in \mathcal{A}_p$, which complies with the boundedness of the Hardy--Littlewood maximal operator $M$ in weighted Lebesgue spaces $L^p_\omega(\R^d)$.
\begin{definition}\label{def:Muckenhoupt}
 \begin{enumerate}
  \item A measurable function $\omega : \mathbb{R}^d \to \mathbb{R}$ is called a \emph{weight function} if $\omega > 0$ a.e.~in $\mathbb{R}^d$.
  \item For each $p \in [1,\infty)$, a weight function $\omega \in L^1_{{\loc}}(\mathbb{R}^d)$ is said to be of the \emph{Muckenhoupt class} $\mathcal{A}_p$, i.e., $\omega \in \mathcal{A}_p$, if there exists a constant $A > 0$ such that for any open ball $B \subset \mathbb{R}^d$,
\begin{alignat*}{2}
&\left( \fint_{B} \omega \,\d x \right) \left( \fint_{B} \omega^{-(p'-1)} \,\d x \right)^{1/(p'-1)} \leq A \quad &&\mbox{ if } \ p\in(1,\infty),\\
&\ M \omega \leq A\omega\ \ \mbox{a.e.~in}\ \mathbb{R}^d\quad &&\mbox{ if }\ p=1.
\end{alignat*}
\item Let $p \in[1,\infty)$ and let $\omega$ be a weight function defined over an open subset $\mathcal{O} \subset \mathbb{R}^d$. Set
\begin{align*}
 L^p_\omega({\mathcal{O}};\mathbb{R}^N) &:= \left\{
 f:{\mathcal{O}} \to \mathbb{R}^N \colon 
 f \mbox{ is measurable and } \int_{{\mathcal{O}}} |f|^p \omega \,\d x < \infty
 \right\},\\
W^{1,p}_\omega({\mathcal{O}};\mathbb{R}^N) &:= \left\{
 f\in L^p_\omega({\mathcal{O}};\mathbb{R}^N) \colon 
 \nabla f \in L^p_\omega({\mathcal{O}};\mathbb{R}^{d\times N})
 \right\}
\end{align*}
equipped with norms
\begin{align*}
 \|f\|_{L^p_\omega} := \left( \int_{\mathcal{O}} |f|^p \omega \,\d x \right)^{1/p},\quad
 \|f\|_{W^{1,p}_\omega} := \|f\|_{L^p_\omega}+\|\nabla f\|_{L^p_\omega},
\end{align*}
respectively. Then they are reflexive Banach spaces and the dual space $(L^p_\omega({\mathcal{O}};\mathbb{R}^N))^*$ of $L^p_\omega({\mathcal{O}};\mathbb{R}^N)$ can be identified with $L^{p'}_\omega({\mathcal{O}};\mathbb{R}^N)$, that is, for any $f \in (L^p_\omega(\mathcal{O};\mathbb{R}^N))^*$, there exists some $v_f \in L^{p'}_\omega(\mathcal{O};\mathbb{R}^N)$ such that
$$
\langle f,u \rangle_{L^p_\omega} = \int_{\mathcal{O}} (v_f \cdot u) \, \omega \,\d x
$$
for all $u \in L^p_\omega(\mathcal{O};\mathbb{R}^N)$.
\end{enumerate}
\end{definition}

We next set up a couple of lemmas for later use. The following lemma is concerned with some relations between maximal functions and Muckenhoupt weights. We refer the reader to~\cite[pp.\,229--230]{21} and~\cite[p.\,5]{22} (see also~\cite[Lemma 2.1]{BuSch16}) for more details. 

\begin{lemma}\label{lem:Muckenhoupt}
Let $f \in L^1_{\rm loc}(\R^d)$ be such that $M f \not\equiv \infty$. Then for all $\alpha \in (0,1)$, it holds that $(Mf)^\alpha \in \mathcal{A}_1$. Moreover, for all $p \in (1,\infty)$ and all $\alpha \in (0,1)$, there holds $(Mf)^{-\alpha (p-1)} \in \mathcal{A}_p$.
\end{lemma}

Let us recall the following lemma, which is concerned with continuous embeddings from $\mathcal{A}_p$-weighted Lebesgue spaces to usual ones (see \cite[pp.\,1125,1126]{BDS}).
\begin{lemma}\label{lem:Emb}
Let $1 < p < \infty$ and $\omega \in \mathcal{A}_p$ and let $D$ be a bounded domain of $\mathbb{R}^d$. Then there exists some $q \in (1,p)$ such that $L^p_\omega(D;\mathbb{R}^N)$ is continuously embedded in $L^q(D;\mathbb{R}^N)$. Hence, $W^{1,p}_\omega(D;\R^N)$ is continuously embedded in $W^{1,q}(D;\R^N)$.
\end{lemma}

Moreover, we have

\begin{lemma}\label{lem:Emb_Mf}
Let $1 < q < p < \infty$ and let $D$ be a bounded domain of $\mathbb{R}^d$. Let $\omega$ be a weight function such that $\omega^{-1} \in L^{q/(p-q)}(D)$. Then $L^p_\omega(D;\mathbb{R}^N)$ is continuously embedded in $L^q(D;\mathbb{R}^N)$.
\end{lemma}

\begin{proof}
Let $u \in L^p_\omega(D;\mathbb{R}^N)$. Since $\omega^{-1}$ belongs to $L^{q/(p-q)}(D)$, we observe that
\begin{equation}\label{Lq-Lpom}
 \int_D |u|^q\, \d x = \int_D |u|^q \omega^{q/p} \omega^{-q/p} \, \d x \leq \|u\|_{L^p_\omega}^q \|\omega^{-1}\|_{L^{q/(p-q)}}^{q/p} < +\infty,
\end{equation}
which implies the desired conclusion.
\end{proof}

\subsection{$\mathcal{A}_p$-weighted biting div-curl lemma}

In this subsection, we recall the \emph{$\mathcal{A}_p$-weighted biting div-curl lemma} developed in~\cite[Theorem 2.6]{BDS}.

\begin{lemma}[$\mathcal{A}_p$-weighted biting div-curl lemma~{\cite[Theorem 2.6]{BDS}}]\label{lem:WDC}
Let $D$ be a bounded open set in $\mathbb{R}^d$, $p \in (1,\infty)$ and $\omega \in \mathcal{A}_p$. For $k \inz_{\geq 1}$, let $(a^k,b^k) : D \to \mathbb{R}^d \times \mathbb{R}^d$ be measurable vector fields complying with the following {\rm (i)--(iii)}\/{\rm :}
\begin{enumerate}
\item[(i)] It holds that
$$
\sup_{k\inz_{\geq 1}} \int_{D} \left( |a^k|^p + |b^k|^{p'} \right) \omega \, \d x < +\infty.
$$
\item[(ii)] It holds that
$$
\int_D \left( a^k_j \partial_{x_i} \varphi - a^k_i \partial_{x_j} \varphi \right) \,\d x = 0 \quad \mbox{ for } \ k \inz_{\geq 1}
$$
for any $\varphi \in W^{1,\infty}_0(D) := \bigcap_{\sigma \in [1,\infty)} W^{1,\sigma}_0(D)$ and $i,j=1,\ldots,d$.
\item[(iii)] Let $(c^k)$ be a sequence in $W^{1,\infty}_0(D)$ such that
$$
\nabla c^k \to 0 \quad \mbox{ weakly star in } L^\infty(D;\mathbb{R}^{d}).
$$
Then
\begin{align*}
\lim_{k \to \infty} \int_D b^k \cdot \nabla c^k \,\d x &= 0.
\end{align*}
\end{enumerate}
Then there exist a subsequence $(k_l)$ of $(k)$, $a \in L^{p}_\omega(D;\mathbb{R}^d)$, $b \in L^{p'}_\omega(D;\mathbb{R}^d)$ and a strictly increasing sequence $(E_j)$ of measurable subsets of $D$ satisfying $|D \setminus E_j| \to 0$ as $j \to \infty$ such that
\begin{alignat*}{4}
a^{k_l} &\to a &&\text{ weakly in } \ L^p_\omega(D;\mathbb{R}^d),\\
b^{k_l} &\to b &&\text{ weakly in } \ L^{p'}_\omega(D;\mathbb{R}^d),\\
(a^{k_l} \cdot b^{k_l}) \, \omega &\to (a \cdot b) \,\omega \quad &&\text{ weakly in } \ L^1(E_j) \quad \text{ for all } \ j \inz_{\geq 1}
\end{alignat*}
as $k_l \to \infty$.
\end{lemma}

The assumption (ii) is stronger than the original one of~\cite[Theorem 2.6]{BDS}, but it is enough to apply the lemma in the present paper.

\subsection{Whitney decomposition of open sets}

We finally recall the Whitney decomposition of general open sets.

\begin{proposition}[Whitney decomposition of general open sets]\label{prop:CubeSplit}
Let $\mathcal{O} \subsetneq \mathbb{R}^d$ be a nonempty open set. Then there exists a countable family of closed, dyadic cubes $(Q_i)$ satisfying all the following {\rm (i)--(vi)}\/{\rm :}
\begin{enumerate}
 \item $\mathcal{O} = \bigcup_i Q_i$ and $Q_i^\circ \cap Q_j^\circ = \emptyset$ if $i \neq j$. Here $Q_i^\circ$ denotes the interior of $Q_i$.
 \item For any $i$, it holds that $\diam (Q_i) < \mathrm{dist} (Q_i,\mathcal{O}^c) \leq 4 \,\diam (Q_i)$. Here and henceforth, $\diam (Q_i)$ stands for the diameter of $Q_i$.
 \item If $Q_i \cap Q_j \neq \emptyset$, then $\frac12 \diam (Q_i) \leq \diam (Q_j) \leq 2 \,\diam (Q_i)$.
 \item For each $i$, it holds that $\sharp \{j \neq i \colon Q_i \cap Q_j \neq \emptyset\} \leq 4^d - 2^d$. Moreover, set $\widetilde{A}_i := \{j \neq i \colon Q_i \cap Q_j \neq \emptyset\}$ and $A_i:=\widetilde{A}_i\cup\{i\}$.
 \item For each $i$, $A_i \ni j$ if and only if $(3/2) Q_i \cap (3/2) Q_j \neq \emptyset$. Here for $k > 0$, $k Q_i := c_i + k (Q_i-c_i)$, where $c_i$ is the center of the cube $Q_i$. 
 \item There exists a partition of unity $(\psi_i)$ in $\mathcal{O}$ such that $\psi_i \in C^{0,1}(\mathbb{R}^d)$, $\chi_{(1/2)Q_i} \leq \psi_i \leq \chi_{(9/8)Q_i}$ {\rm (}i.e., $\supp \psi_i \subset (9/8) Q_i${\rm )} and there exists a constant $c(d)$ such that $\diam(Q_i)|\nabla \psi_i| \leq c(d)$ uniformly for $i$. Furthermore, it holds that
\begin{align*}
 \sum_i \psi_i(x) = \begin{cases}
		     1 &\mbox{ if } \ x \in \mathcal{O},\\
		     0 &\mbox{ if } \ x \notin \mathcal{O}.
		    \end{cases}
\end{align*}
\end{enumerate}
\end{proposition}
\color{black}

We refer the reader to~\cite[Proposition 3.17]{Dien1} and~\cite[Proposition 3.1]{BuSch16} for proofs of (i)--(iv) and (v)--(vi), respectively.

\section{Weak solutions}\label{S:weak}

In this section, we discuss existence and uniqueness of weak solutions to \P \ for $f \in L^p(\Omega;\mathbb{R}^{d\times N})$ and $g \in L^r(\Omega;\mathbb{R}^N)$ with the aid of a monotone operator theory.

\begin{proposition}\label{prop:approxsol}
For $f \in L^p(\Omega;\mathbb{R}^{d\times N})$ and $g \in L^r(\Omega;\mathbb{R}^N)$, the Dirichlet problem \P \ admits a unique weak solution $u \in X^{p,r}(\Omega;\mathbb{R}^N)$.
\end{proposition}

To prove this, we first note the following basic properties of the space $X^{p,r}(\Omega;\mathbb{R}^N)$, which can be proved in a standard manner (see, e.g.,~\cite{Bre}):

\begin{proposition}\label{prop:Xp,r}
\begin{enumerate}
 \item[(i)] The space $X^{p,r}(\Omega;\mathbb{R}^N)$ is a reflexive Banach space.
 \item[(ii)] It holds that $W^{1,p}_0(\Omega;\mathbb{R}^N) \cap L^r(\Omega;\mathbb{R}^N) \subset X^{p,r}(\Omega;\mathbb{R}^N)$.
 \item[(iii)] For any $F \in (X^{p,r}(\Omega;\mathbb{R}^N))^*$, there exist $F_1 \in L^{p'}(\Omega;\mathbb{R}^{d\times N})$ and $F_2\in L^{r'}(\Omega;\mathbb{R}^N)$ such that
$$
\langle F,u\rangle_{X^{p,r}} = \int_\Omega F_1 : \nabla u \, \d x + \int_\Omega F_2 \cdot u \, \d x
$$
for $u \in X^{p,r}(\Omega;\mathbb{R}^N)$.
\end{enumerate}
\end{proposition}

In particular, in order to check the assertion (ii), one may observe that $W^{1,p}_0(\Omega) \cap L^r(\Omega)$ coincides with the closure of $C^\infty_c(\Omega)$ in $W^{1,p}(\Omega) \cap L^r(\Omega)$ (cf.~see~\cite[Proposition 9.18]{Bre}).

\prf{
\begin{proof}
As for (i), we first prove the completeness of $(X^{p,r}(\Omega;\mathbb{R}^N), \|\cdot\|_{X^{p,r}})$. Let $(u^k)$ be a Cauchy sequence in $X^{p,r}(\Omega;\mathbb{R}^N)$. Then $(u^k)$ and $(\nabla u^k)$ are Cauchy sequences in $L^{r}(\Omega;\mathbb{R}^N)$ and $L^{p}(\Omega;\mathbb{R}^{d\times N})$, respectively. From the completeness of $L^{r}(\Omega;\mathbb{R}^N)$ and $L^{p}(\Omega;\mathbb{R}^{d\times N})$, there exist $u \in L^{r}(\Omega;\mathbb{R}^N)$ and $\xi \in L^{p}(\Omega;\mathbb{R}^{d\times N})$ such that
\begin{alignat*}{3}
 u^k &\to u \quad &&\text{ in } L^{r}(\Omega;\mathbb{R}^N),\\
 \nabla u^k &\to \xi \quad &&\text{ in } L^{p}(\Omega;\mathbb{R}^{d\times N}).
\end{alignat*}
Moreover, we observe that
\begin{align*}
 \int_\Omega \nabla u^k \cdot \varphi \,\d x &= -\int_\Omega u^k \cdot \nabla \varphi \,\d x
\end{align*}
for any $\varphi \in C_c^\infty(\Omega;\mathbb{R}^N)$. Passing to the limit as $k \to \infty$, we deduce that
$$
\int_\Omega \xi \cdot \varphi \,\d x = -\int_\Omega u \cdot \nabla \varphi \,\d x,
$$
which yields $\xi = \nabla u$. Therefore we infer that $u \in X^{p,r}(\Omega;\mathbb{R}^N)$ and $u^k \to u$ in $X^{p,r}(\Omega;\mathbb{R}^N)$. Thus we conclude that $X^{p,r}(\Omega;\mathbb{R}^N)$ is complete.

We next prove the reflexivity. Set $E := L^p(\Omega;\mathbb{R}^{d\times N}) \times L^r(\Omega;\mathbb{R}^N)$ equipped with the norm $\|\cdot\|_E$ given by
$$
\|h\|_E:=\|h_1\|_{L^p(\Omega;\mathbb{R}^{d\times N})}+\|h_2\|_{L^r(\Omega;\mathbb{R}^N)}
$$
for $h =(h_1,h_2) \in E$. Moreover, let $T: X^{p,r}(\Omega;\mathbb{R}^N) \to E$ be defined by $Tu := (\nabla u,u)$ for $u \in X^{p,r}(\Omega;\mathbb{R}^N)$. Then $T$ is isometric and hence injective. Moreover, one observes that
$$
E^{**} = L^p(\Omega;\mathbb{R}^{d\times N})^{**} \times L^r(\Omega;\mathbb{R}^N)^{**} = L^p(\Omega;\mathbb{R}^{d\times N})\times L^r(\Omega;\mathbb{R}^N) = E.
$$
Thus $E$ is reflexive. Therefore the image $T(X^{p,r}(\Omega;\mathbb{R}^N))$ is closed and hence reflexive. Since $T$ is isometric, $X^{p,r}(\Omega;\mathbb{R}^N)$ turns out to be reflexive. Thus (i) is proved.

We next prove (ii). As in the proof of Proposition 9.18 of~\cite{Bre}, one can verify that $W^{1,p}_0(\Omega) \cap L^r(\Omega)$ coincides with the closure of $C^\infty_c(\Omega)$ in $W^{1,p}(\Omega) \cap L^r(\Omega)$, whence (ii) follows immediately from the definition of $X^{p,r}(\Omega;\mathbb{R}^N)$.

Finally, we prove (iii). Let $F \in (X^{p,r}(\Omega;\mathbb{R}^N))^*$ be fixed. Set $E := L^p(\Omega;\mathbb{R}^{d\times N}) \times L^r(\Omega;\mathbb{R}^N)$ be equipped with the norm $\|\,\cdot\,\|_E$ given by 
$$
\|h\|_E:=\|h_1\|_{L^p(\Omega;\mathbb{R}^{d\times N})}+\|h_2\|_{L^r(\Omega;\mathbb{R}^N)}
$$
for $h = (h_1,h_2) \in E$. Moreover, define a mapping $T : X^{p,r}(\Omega;\mathbb{R}^N) \to E$  by $Tu := (\nabla u,u)$ for $u \in X^{p,r}(\Omega;\mathbb{R}^N)$. Then it is isometric (and hence injective). Set $G := T(X^{p,r}(\Omega;\mathbb{R}^N)) \subset E$. Then $T : X^{p,r}(\Omega;\mathbb{R}^N) \to G$ is of course bijective. Let $S$ be its inverse mapping and define a mapping $\Phi_F : G \to \mathbb{R}$ by 
$$
\Phi_F(h) := \langle F, Sh \rangle_{X^{p,r}}
$$
for $h = (\nabla h_0,h_0) \in G$. Then it follows that
\begin{align*}
|\Phi_F(h)|
 &\leq \|F\|_{(X^{p,r})^*} \|Sh\|_{X^{p.r}}\\
 &\leq \|F\|_{(X^{p,r})^*} \|S\|_{\mathcal{L}(G,X^{p,r})} \|h\|_E\\
 &=\|F\|_{(X^{p,r})^*} \|h\|_E
\end{align*}
for $h = (\nabla h_0,h_0) \in G$. Since $\Phi_F$ is linear by definition, one observes that $\Phi_F \in G^*$. Thus due to Hahn--Banach's theorem, we can extend $\Phi_F$ onto $E$ and denote the extension by $\overline{\Phi_F}$. Then it follows that 
$$
\overline{\Phi_F} \in E^* \simeq L^{p'}(\Omega;\mathbb{R}^{d\times N}) \times L^{r'}(\Omega;\mathbb{R}^{N}).
$$
Thus as in~\cite[Proposition 8.14]{Bre}, one can verify that
\begin{align*}
\langle F,u\rangle_{X^{p,r}}
 &= \langle F, (S\circ T)u \rangle_{X^{p,r}}\\
 &=\langle \Phi_F, Tu \rangle_{X^{p,r}}\\
 &=\langle \overline{\Phi_F}, (\nabla u,u)_E \rangle_{X^{p,r}}\\
 &=\langle \overline{\Phi_F}(\cdot,0), \nabla u \rangle_{L^{p}(\Omega;\mathbb{R}^{d\times N})} + \langle \overline{\Phi_F}(0,\cdot), u \rangle_{L^{r}(\Omega;\mathbb{R}^N)}\\
 &= \int_\Omega F_1 \cdot \nabla u \,\d x + \int_\Omega F_2 \cdot u \,\d x
\end{align*}
for any $u \in X^{p,r}(\Omega;\mathbb{R}^N)$. Thus (iii) is proved.
\end{proof}
}

We set $\mathcal{A}u := -\,\dv A \,(\cdot,\nabla u)$ and $\mathcal{B}u := \mathcal{A}u + |u|^{r-2}u$ for $u \in X^{p,r}(\Omega;\mathbb{R}^N)$. Then the operator $\mathcal{B} : X^{p,r}(\Omega;\mathbb{R}^N) \to (X^{p,r}(\Omega;\mathbb{R}^N))^*$ is defined by
\begin{align*}
[ \mathcal{B}u ](\varphi)
 &:=\int_\Omega A(\cdot,\nabla u) : \nabla \varphi \,\d x + \int_\Omega |u|^{r-2}u \cdot \varphi \,\d x
\end{align*}
for $u, \varphi \in X^{p,r}(\Omega;\mathbb{R}^N)$; indeed, we note that
\begin{align*}
|[ \mathcal{B}u ](\varphi)|
 &\leq \|A(\cdot,\nabla u)\|_{L^{p'}} \|\nabla\varphi\|_{L^p} + \||u|^{r-1}\|_{L^{r'}} \|\varphi\|_{L^r}\\
 &\leq (\|A(\cdot,\nabla u)\|_{L^{p'}} + \|u\|^{r-1}_{L^{r}}) \|\varphi\|_{X^{p,r}}
\end{align*}
for $\varphi \in X^{p,r}(\Omega;\mathbb{R}^N)$, and hence, we find that $\mathcal{B}u \in (X^{p,r}(\Omega;\mathbb{R}^N))^*$. Moreover, we can immediately observe that

\begin{lemma}
The operator $\mathcal{B} : X^{p,r}(\Omega;\mathbb{R}^N) \to (X^{p,r}(\Omega;\mathbb{R}^N))^*$ is a strictly monotone operator.
\end{lemma}

\prf{
\begin{proof}
For any $n \inz_{\geq 1}$, $1 < t < \infty$ and $v,w \in L^t(\Omega;\mathbb{R}^n)$, due to Young's inequality, we see that
\begin{align*}
\MoveEqLeft{
\langle |v|^{t-2}v-|w|^{t-2}w,v-w\rangle_{L^t}
}\\
 &= \|v\|^t_{L^t} + \|w\|^t_{L^t} - \langle |v|^{t-2}v, w \rangle_{L^t} - \langle v, |w|^{t-2}w \rangle_{L^t}\\
 &\geq \|v\|^t_{L^t} + \|w\|^t_{L^t} - \left( \frac{\|v\|^t_{L^t}}{t'} + \frac{\|w\|^t_{L^t}}{t} \right) - \left( \frac{\|v\|^t_{L^t}}{t} + \frac{\|w\|^t_{L^t}}{t'} \right)\\
 &=0.
\end{align*}
Hence for any $u_1,u_2\in X^{p,r}(\Omega;\mathbb{R}^N)$, setting $(n,t,v,w) := (d\times N,p,\nabla u_1,\nabla u_2)$ and applying the above to $(N,r,u_1,u_2)$, we obtain the monotonicity of $\mathcal{A}$ and $u \mapsto |u|^{r-2}u$. Thus $\mathcal{B}$ turns to be monotone.
\end{proof}
}

Furthermore, we have

\begin{lemma}
The operator $\mathcal{B} : X^{p,r}(\Omega;\mathbb{R}^N)\to(X^{p,r}(\Omega;\mathbb{R}^N))^*$ is coercive and continuous.
\end{lemma}

\begin{proof}
We first prove the coercivity. Let $(u^k)$ be a sequence in $X^{p,r}(\Omega;\mathbb{R}^N)$ and let $C \geq 0$ be a constant such that
\begin{equation}\label{coer2}
\frac{\langle \mathcal{B}u^k,u^k\rangle_{X^{p,r}}}{\|u^k\|_{X^{p,r}}} \leq C \quad \mbox{ for } \ k \inz_{\geq 1} \ \mbox{(large enough)},
\end{equation}
which implies
\begin{align*}
 C_1 \|\nabla u^k\|_{L^p}^p + \|u^k\|_{L^r}^r
 &\stackrel{\eqref{ineq:ellipAcoer}}{\leq} \langle \mathcal{B}u^k, u^k \rangle_{X^{p,r}} + \|\beta_1\|_{L^1}\\
 &\stackrel{\eqref{coer2}}{\leq} C \|u^k\|_{X^{p,r}} +\|\beta_1\|_{L^1}\\
 &\leq \frac{C_1}{2}\|\nabla u^k\|_{L^p}^p + \frac{1}{2}\|u^k\|_{L^r}^r + C.
\end{align*}
Thus $\|u^k\|_{X^{p,r}}$ is bounded, and therefore, $\mathcal{B}$ turns out to be coercive.

We claim that the operator $\mathcal{B}:X^{p,r}(\Omega;\mathbb{R}^N)\to (X^{p,r}(\Omega;\mathbb{R}^N))^*$ is continuous. Indeed, let $u^k \to u$ in $X^{p,r}(\Omega;\mathbb{R}^N)$. Then $\nabla u^k \to \nabla u$ in $L^{p}(\Omega;\mathbb{R}^{d\times N})$ and $u^k \to u$ in $L^{r}(\Omega;\mathbb{R}^N)$. Then $A(\cdot,\nabla u^k) \to A(\cdot,\nabla u)$ strongly in $L^{p'}(\Omega;\mathbb{R}^{d\times N})$ and $|u^k|^{r-2}u^k \to |u|^{r-2}u$ strongly in $L^{r'}(\Omega;\mathbb{R}^N)$. Hence we observe that
\begin{align*}
 \MoveEqLeft{
 \|\mathcal{B}u^k-\mathcal{B}u\|_{(X^{p,r}(\Omega;\mathbb{R}^N))^*}
 }\\
\prf{ &= \sup_{\|\varphi\|_{X^{p,r}}=1} \langle \mathcal{B}u^k - \mathcal{B}u, \varphi\rangle_{X^{p,r}}\\
 &= \sup_{\|\varphi\|_{X^{p,r}}=1} \left[ \int_\Omega (A(\cdot,\nabla u^k)-A(\cdot,\nabla u)) : \nabla\varphi \,\d x\right.\\
 &\qquad \left. + \int_\Omega (|u^k|^{r-2}u^k-|u|^{r-2}u) \cdot \varphi \,\d x \right]\\
 &\leq \sup_{\|\varphi\|_{X^{p,r}}=1} \left( \|A(\cdot,\nabla u^k)-A(\cdot,\nabla u)\|_{L^{p'}} \right.\\
 &\qquad \left.+ \||u^k|^{r-2}u^k-|u|^{r-2}u\|_{L^{r'}} \right) \|\varphi\|_{X^{p,r}}\\}
 &\leq \|A(\cdot,\nabla u^k)-A(\cdot,\nabla u)\|_{L^{p'}} + \||u^k|^{r-2}u^k-|u|^{r-2}u\|_{L^{r'}} \to 0
\end{align*}
as $k \to\infty$. Hence $\mathcal{B}$ is continuous.
\end{proof}

We are now in a position to prove Proposition \ref{prop:approxsol}.

\begin{proof}[Proof of Proposition {\rm \ref{prop:approxsol}}]
Let $f \in L^p(\Omega;\mathbb{R}^{d\times N})$ and $g \in L^r(\Omega;\mathbb{R}^N)$. Then we observe that
\begin{align*}
\MoveEqLeft{
 \langle -\,\dv \,(|f|^{p-2}f)+|g|^{r-2}g, \varphi \rangle_{X^{p,r}}
}\\
 &:= \int_\Omega |f|^{p-2}f : \nabla\varphi \,\d x + \int_\Omega |g|^{r-2}g \cdot \varphi \, \d x
\end{align*}
for $\varphi \in X^{p,r}(\Omega;\mathbb{R}^N)$. Hence one can deduce that $-\,\dv \,(|f|^{p-2}f) + |g|^{r-2}g \in (X^{p,r}(\Omega;\mathbb{R}^N))^*$. Due to the monotonicity and continuity of $\mathcal{B}$, it turns out to be maximal monotone, and moreover, from the coercivity, we infer that $\mathcal{B}$ is bijective (see, e.g.,~\cite{B}). Therefore \P \ admits a weak solution $u \in X^{p,r}(\Omega;\mathbb{R}^N)$. The uniqueness follows immediately from the strict monotonicity.
\end{proof}

\section{Weighted local energy estimates}\label{S:loc_en_est}

This section is devoted to deriving weighted local energy estimates for weak solutions to \P \ by employing a \emph{relative truncation technique} developed in~\cite{BuSch16}, where weighted energy estimates are established for a bounded domain case along with globally integrable forcing, i.e., $f \in L^q(\Omega;\R^N)$, and therefore, neither localization nor absorption is needed.

\subsection{Relative truncation}

Let $\mathcal{O} \subset \mathbb{R}^d$ be an open set. For $u \in L^1_\loc(\mathbb{R}^d;\mathbb{R}^N)$, define the \emph{relative truncation} $u_{\mathcal{O}} : \mathbb{R}^d \to \mathbb{R}$ of $u$ into the open set $\mathcal{O}$ by 
\begin{align*}
 u_{\mathcal{O}}(x) 
 &:= \begin{cases}
      \sum_i \psi_i(x) \overline{u}_i &\mbox{if } \ x \in \mathcal{O}, \\
      u(x) &\mbox{if } \ x \notin \mathcal{O},
     \end{cases}
\\
 \overline{u}_i
 &:= \begin{cases}
      \fint_{(9/8)Q_i} u\,\d x & \mbox{if } \ (9/8)\color{black}Q_i \subset \Omega, \\
      0 & \mbox{otherwise.}
     \end{cases}
\end{align*}
Here $(Q_i)$ and $(\psi_i)$ are the cubes and partition of unity introduced in Proposition \ref{prop:CubeSplit} for the open set $\mathcal{O}$. In what follows, we simply use the same notation $u_{\mathcal{O}}$ for the restriction of $u_{\mathcal{O}}$ into the domain $\Omega$. Then we recall
\begin{lemma}[{\cite[Lemma 3.2]{BuSch16}}]\label{lem:BS3.2}
Let $\Omega$ be a Lipschitz domain of $\mathbb{R}^d$ and let $u \in W^{1,p}_0(\Omega;\mathbb{R}^N)$ for some $p \in [1,\infty)$. Then $u_{\mathcal{O}} \in W^{1,p}_0(\Omega;\mathbb{R}^N)$ and there exists a constant $c = c(d,p,\Omega)$ such that
\begin{align*}
\int_{\Omega} |\nabla(u-u_{\mathcal{O}})|^p \,\d x \leq c \int_{\mathcal{O}\cap\Omega} |\nabla u|^p \,\d x.
\end{align*}
\end{lemma}

Moreover, we have

\begin{proposition}\label{P:rel_tr_Xpr}
For each $u \in X^{p,r}(\Omega;\mathbb{R}^N)$ and $\rho \in C^\infty_c(\R^d)$, the relative truncation $(u{\rho})_{\mathcal{O}}$ of the product $u{\rho}$ into $\mathcal{O}$ belongs to $X^{p,r}(\Omega;\mathbb{R}^N)$. 
\end{proposition}

\begin{proof}
We first note that the zero extension of $u$ onto $\mathbb{R}^d$ belongs to $X^{p,r}(\mathbb{R}^d;\mathbb{R}^N)$ (it will still be denoted by $u$ simply). Let $x \in \mathcal{O}$ and let $i$ be such that $x \in Q_i$ (see Proposition \ref{prop:CubeSplit}). Then we observe from Proposition \ref{prop:CubeSplit} that
\begin{align*}
|(u{\rho})_{\mathcal{O}}(x)|
 &\leq \sum_{j \in A_i} \psi_j(x)|\overline{(u{\rho})}_j|\\
 &\leq \sum_{j\in A_i} \psi_j(x)\frac{2^d}{|Q_i|}\int_{(9/8)Q_j\cap\Omega}|u|{\rho} \,\d x\\
 &\leq \frac{C}{|Q_i|} \sum_{j\in A_i}\int_{(9/8)Q_j\cap\Omega}|u| \,\d x\\
 &\leq \frac{C}{|Q_i|} \sum_{j\in A_i}\|u\|_{L^r((9/8)Q_j\cap\Omega)}|(9/8)Q_j|^{1/r'}\\
 &\leq \frac{C}{|Q_i|} \sum_{j\in A_i}\|u\|_{L^r((9/8)Q_j\cap\Omega)}|(9/4)Q_i|^{1/r'}\\
 &\leq C|Q_i|^{-1/r} \sum_{j\in A_i}\|u\|_{L^r((9/8)Q_j\cap\Omega)}.
\end{align*}
Therefore recalling that $\mathcal{O} = \bigcup_i Q_i$, we deduce that
\begin{align}
 \int_{\mathcal{O}} |(u{\rho})_{\mathcal{O}}(x)|^r \,\d x
 &= \sum_{i} \int_{Q_i} |(u{\rho})_{\mathcal{O}}(x)|^r \,\d x\notag\\
 &\leq \sum_{i} \int_{Q_i} C |Q_i|^{-1} \left| \sum_{j\in A_i} \|u\|_{L^r((9/8)Q_j\cap\Omega)} \right|^r \,\d x\notag\\
 &\leq C\sum_{i} |A_i|^{r-1} \left( \sum_{j\in A_i} \|u\|^r_{L^r((9/8)Q_j\cap\Omega)} \right)\notag\\
 &\leq 4^{d(r-1)} C\sum_{i} \left( \sum_{j\in A_i} \|u\|^r_{L^r((9/8)Q_j\cap\Omega)} \right)\notag\\
 &\stackrel{(*)}= 4^{d(r-1)} C\sum_{j} \left( \sum_{i \in A_j} \|u\|^r_{L^r((9/8)Q_j\cap\Omega)} \right)\notag\\
 &\stackrel{(\star)}\leq 4^{d(r+1)} C \int_{\mathcal{O}\cap\Omega} |u|^r \,\d x < +\infty.\label{poco}
\end{align}
Here the equality $(*)$ is derived from the following fundamental observation:
\begin{align*}
\lefteqn{
\sum_{i} \left( \sum_{j\in A_i} \|u\|^r_{L^r((9/8)Q_j\cap\Omega)} \right)
}\\
&= \sum_{i} \left( \sum_{j} \chi_{A_i}(j) \|u\|^r_{L^r((9/8)Q_j\cap\Omega)} \right)\\
&= \sum_{i} \left( \sum_{j} \chi_{A_j}(i) \|u\|^r_{L^r((9/8)Q_j\cap\Omega)} \right)\\
&= \sum_{j} \left( \sum_{i} \chi_{A_j}(i) \|u\|^r_{L^r((9/8)Q_j\cap\Omega)} \right)\\
&= \sum_{j} \left( \sum_{i \in A_j} \|u\|^r_{L^r((9/8)Q_j\cap\Omega)} \right),
\end{align*}
where we used the equivalence $i \in A_j \Leftrightarrow j \in A_i$ as well as Fubini's lemma for double series of nonnegative terms. Moreover, the inequality $(\star)$ also follows (similarly to the above observation) from Proposition \ref{prop:CubeSplit}, which implies that $(9/8)Q_j$ intersects with only $Q_k$ for $k \in A_j$. Indeed, we note that
\begin{align*}
\sum_{j} \|u\|^r_{L^r((9/8)Q_j\cap\Omega)}
&\leq \sum_{j} \sum_{k \in A_j} \|u\|^r_{L^r(Q_k\cap\Omega)}\\
&= \sum_{k} \sum_{j \in A_k} \|u\|^r_{L^r(Q_k\cap\Omega)}\\
&\leq 4^d \sum_{k} \|u\|^r_{L^r(Q_k\cap\Omega)} = 4^d \|u\|^r_{L^r(\mathcal{O}\cap\Omega)}.
\end{align*}
Hence it follows that
\begin{align*}
 \int_{\Omega} |(u{\rho})_{\mathcal{O}}|^r \,\d x
 &= \int_{\mathcal{O}} |(u{\rho})_{\mathcal{O}}|^r \,\d x + \int_{\Omega\backslash\mathcal{O}}|u{\rho}|^r \,\d x < +\infty.
\end{align*}
Thus we have obtained $(u{\rho})_{\mathcal{O}} \in L^r(\Omega;\mathbb{R}^N)$.

Let $(u_n)$ be a sequence in $C^\infty_c(\Omega;\R^N)$ such that $u_n \to u$ strongly in $L^r(\Omega;\R^N)$ and $\nabla u_n \to \nabla u$ strongly in $L^p(\Omega;\R^{d\times N})$. Since $\supp {\rho}$ is compact, we find that $u_n {\rho} \to u \rho$ strongly in $L^p(\Omega;\mathbb{R}^N)$ and
$$
\nabla(u_n{\rho}) = (\nabla u_n){\rho} + u_n\nabla{\rho} \to (\nabla u){\rho} + u\nabla{\rho} = \nabla (u \rho)
$$
strongly in $L^p(\Omega;\mathbb{R}^{d\times N})$. Thus we observe that $u {\rho} \in W_0^{1,p}(\Omega;\mathbb{R}^N)$, and therefore, by the use of Lemma \ref{lem:BS3.2}, we deduce that $(u{\rho})_{\mathcal{O}} \in W_0^{1,p}(\Omega;\mathbb{R}^N)$. Combining all these facts, we have $(u{\rho})_{\mathcal{O}} \in W_0^{1,p}(\Omega;\mathbb{R}^N) \cap L^r(\Omega;\mathbb{R}^N)$, and therefore, Proposition \ref{prop:Xp,r} yields $(u{\rho})_{\mathcal{O}} \in X^{p,r}(\Omega;\mathbb{R}^N)$. The proof is completed.
\end{proof}

\subsection{Weighted local energy estimates}

Developing the argument used in~\cite[Proposition 4.1]{BuSch16}, we shall prove the following proposition, which will play a key role for proving Theorem \ref{thm:elliplocDS}.

\begin{proposition}\label{prop:LocEstimate}
Let $\Omega$ be a {\rm (}possibly unbounded{\rm )} Lipschitz domain of $\mathbb{R}^d$ and let $1 < p < r < \infty$. Then there exists a constant $\varepsilon_1 = \varepsilon_1(C_1,C_2,\Omega,d,N,p,r)> 0$ satisfying the following\/{\rm :} Let $f \in L^p(\Omega;\mathbb{R}^{d\times N})$ and $g \in L^r(\Omega;\mathbb{R}^{N})$ and let $u \in X^{p,r}(\Omega;\mathbb{R}^N)$ be the unique weak solution to \P. Let $\varepsilon \in (0,\varepsilon_1)$ and $h \in L^1_{{\loc}}(\overline{\Omega}) \setminus \{0\}$ be arbitrarily fixed. Then there exists a constant $C_0 = C_0(C_1,C_2,\Omega,d,N,p,r,\varepsilon) > 0$ such that
\begin{align*}
\MoveEqLeft{
\int_{\Omega_{R}} |\nabla u|^{p} \omega \,\d x + \int_{\Omega_{R}} |u|^{r} \omega \,\d x
}\\
&\leq C_0 \left( 
 \int_{\Omega_{2R}} |f|^{p} \omega \,\d x + \int_{\Omega_{2R}} |g|^r \omega \,\d x \right.\\
&\qquad + \left. \int_{\Omega_{2R}} \beta_1 \omega \,\d x + \int_{\Omega_{2R}}\beta_2^{p'} \omega \,\d x + \delta^{-\vep} R^{d-\frac{pr}{r-p}} \right),
\end{align*}
where $\omega$ is the weight function given by $\omega = (M[ \bar h \chi_{\Omega_{2R}} ] + \delta)^{-\varepsilon}$ with the zero extension $\bar h$ of $h$ onto $\R^d$ and $\beta_1 \in L^1_{\loc}(\overline{\Omega})$ and $\beta_2\in L^{p'}_{\loc}(\overline{\Omega})$ are nonnegative functions satisfying \eqref{ineq:ellipAcoer} and \eqref{ineq:ellipAbdd}, for every $R > 0$ and $\delta \in (0,1]$.
\end{proposition}

\begin{proof}
In what follows, we use the same letters $u,f,h$ for their zero extensions onto $\mathbb{R}^d$ if no confusion can arise. Moreover, let $R > 0$ and $0 < \delta \leq 1$ be arbitrarily fixed and set $h_\delta := |h|\chi_{\Omega_{2R}} + \delta$. Moreover, we take an open subset of $\R^d$,
$$
\mathcal{O}_\lambda := [Mh_\delta>\lambda] := \left\{ x \inr^d \colon Mh_\delta(x) > \lambda \right\}.
$$
Here and henceforth, we simply denote by $v_\lambda$ the relative truncation $v_{\mathcal{O}_\lambda}$ of each function $v$ into $\mathcal{O}_\lambda$.

Let $R' \in (R,2R]$ and let ${\rho}$ be a cut-off function complying with 
\begin{equation}\label{cut-off0}
\left\{
\begin{aligned}
&{\rho} \in C^\infty_c(\R^d), \quad 0 \leq {\rho} \leq 1 \ \mbox{ in } \mathbb{R}^d, \quad {\rho} \equiv 1 \ \mbox{ on } B_R,\\
&\supp {\rho} \subset \overline{B_{R'}}, \quad \sup_{x\inr^d} |\nabla{\rho}(x)| \leq C/(R'-R)
\end{aligned}
\right. 
\end{equation}
for some constant $C > 0$ independent of $R'$ and $R$. Thanks to Proposition \ref{P:rel_tr_Xpr}, we see that $(u{\rho})_\lambda {\rho}^{2r-1} \in X^{p,r}(\Omega;\mathbb{R}^N)$, and hence, it follows from the weak form that
\begin{align}\label{weakform}
\MoveEqLeft{
\int_{\mathcal{O}^c_\lambda\cap\Omega_{R'}} A(\cdot,\nabla u) : \nabla (u{\rho}^{2r}) \,\d x + \int_{\mathcal{O}^c_\lambda\cap\Omega_{R'}} |u|^{r-2}u \cdot (u{\rho}^{2r}) \,\d x
}\notag\\
 &= \int_{\mathcal{O}_\lambda\cap\Omega_{R'}} |f|^{p-2}f : \nabla ((u{\rho})_\lambda{\rho}^{2r-1}) \,\d x + \int_{\mathcal{O}^c_\lambda\cap\Omega_{R'}} |f|^{p-2}f : \nabla (u{\rho}^{2r}) \,\d x\notag\\
 &\quad +\int_{\mathcal{O}_\lambda\cap\Omega_{R'}} |g|^{r-2}g \cdot ((u{\rho})_\lambda{\rho}^{2r-1}) \,\d x + \int_{\mathcal{O}^c_\lambda\cap\Omega_{R'}} |g|^{r-2}g \cdot (u{\rho}^{2r}) \,\d x\notag\\
 &\quad -\int_{\mathcal{O}_\lambda\cap\Omega_{R'}} A(\cdot,\nabla u) : \nabla ((u{\rho})_\lambda{\rho}^{2r-1}) \,\d x\notag\\
 &\quad -\int_{\mathcal{O}_\lambda\cap\Omega_{R'}} |u|^{r-2}u \cdot ((u{\rho})_\lambda{\rho}^{2r-1}) \,\d x.
\end{align}
We then observe that
\begin{align*}
 \MoveEqLeft{
 \mbox{(the left-hand side of \eqref{weakform})}
}\notag\\
 &= \int_{\mathcal{O}^c_\lambda\cap\Omega_{R'}} A(\cdot,\nabla u) : \left[ (\nabla u) {\rho}^{2r} + u \otimes 2r{\rho}^{2r-1} \nabla {\rho} \right] \,\d x\notag\\
 &\quad + \int_{\mathcal{O}^c_\lambda\cap\Omega_{R'}} |u|^{r} {\rho}^{2r} \,\d x\notag\\
 &\geq \int_{\mathcal{O}^c_\lambda\cap\Omega_{R'}} \left[ (C_1|\nabla u|^p - \beta_1) {\rho}^{2r} + 2r A(\cdot,\nabla u) : (\nabla{\rho}) \otimes u {\rho}^{2r-1} \right] \,\d x\notag\\
 &\quad + \int_{\mathcal{O}^c_\lambda\cap\Omega_{R'}} |u|^{r} {\rho}^{2r} \,\d x.\notag
\end{align*}
Hence noting that $2(r-1)p'>2(r-1)r'=2r$ and $0 \leq {\rho} \leq 1$ in $\mathbb{R}^d$ and employing Young's inequality for any $\delta_1,\delta_2 > 0$, which will be chosen later, we derive from \eqref{weakform} that
\begin{align}\label{int I,J}
 \MoveEqLeft{
 \int_{\mathcal{O}^c_\lambda\cap\Omega_{R'}} (C_1|\nabla u|^p-\beta_1) {\rho}^{2r} \,\d x + \int_{\mathcal{O}^c_\lambda\cap\Omega_{R'}} |u|^{r} {\rho}^{2r} \,\d x
 }\notag\\%
 &\leq 2r\int_{\mathcal{O}^c_\lambda\cap\Omega_{R'}} |A(\cdot,\nabla u)| {\rho}^{2r-1} |u||\nabla{\rho}| \,\d x \notag\\
 &\quad + \int_{\mathcal{O}_\lambda\cap\Omega_{R'}} |f|^{p-1} |\nabla (u{\rho})_\lambda| {\rho}^{2r-1} \,\d x \notag \\
 &\quad + (2r-1) \int_{\mathcal{O}_\lambda\cap\Omega_{R'}} |f|^{p-1} {\rho}^{2r-2} |(u{\rho})_\lambda| |\nabla{\rho}| \,\d x \notag \\
 &\quad + \int_{\mathcal{O}^c_\lambda\cap\Omega_{R'}} |f|^{p-1} |\nabla u| {\rho}^{2r} \,\d x
 + 2r \int_{\mathcal{O}^c_\lambda\cap\Omega_{R'}} |f|^{p-1} {\rho}^{2r-1} |u||\nabla{\rho}| \,\d x \notag\\ 
 &\quad + \int_{\mathcal{O}_\lambda\cap\Omega_{R'}} |g|^{r-1} |(u{\rho})_\lambda| {\rho}^{2r-1} \,\d x\notag\\
 &\quad + \int_{\mathcal{O}^c_\lambda\cap\Omega_{R'}} |g|^{r-1} |u| {\rho}^{2r} \,\d x + \int_{\mathcal{O}_\lambda\cap\Omega_{R'}} |A(\cdot,\nabla u)| |\nabla (u{\rho})_\lambda| {\rho}^{2r-1} \,\d x\notag\\
 &\quad + (2r-1) \int_{\mathcal{O}_\lambda\cap\Omega_{R'}} |A(\cdot,\nabla u)| |(u{\rho})_\lambda| {\rho}^{2r-2} |\nabla{\rho}| \,\d x\notag\\
 &\quad +\int_{\mathcal{O}_\lambda\cap\Omega_{R'}} |u|^{r-1} |(u{\rho})_\lambda| {\rho}^{2r-1} \,\d x\notag\\
 &\leq 2r \int_{\mathcal{O}^c_\lambda\cap\Omega_{R'}} (C_2|\nabla u|^{p-1}+\beta_2) {\rho}^{2r-1} |u| |\nabla{\rho}| \,\d x \notag\\
 &\quad + \int_{\mathcal{O}^c_\lambda\cap\Omega_{R'}} (|f|^{p-1}{\rho}^{2r/p'}) (|\nabla u| {\rho}^{2r/p}) \,\d x + 2r \int_{\mathcal{O}^c_\lambda\cap\Omega_{R'}} (|f|^{p-1} {\rho}^{2r-1}) (|u||\nabla{\rho}|) \,\d x \notag\\
 &\quad + \int_{\mathcal{O}^c_\lambda\cap\Omega_{R'}} (|g|^{r-1}{\rho}^{2r/r'})(|u|{\rho}^{2r/r}) \,\d x\notag\\
 &\quad + \int_{\mathcal{O}_\lambda\cap\Omega_{R'}} (|f|^{p-1} + C_2|\nabla u|^{p-1}+\beta_2) |\nabla (u{\rho})_\lambda| {\rho}^{2r-1} \,\d x\notag\\
 &\quad + (2r-1) \int_{\mathcal{O}_\lambda\cap\Omega_{R'}} (|f|^{p-1}+C_2|\nabla u|^{p-1}+\beta_2) {\rho}^{2r-2} |(u{\rho})_\lambda| |\nabla{\rho}| \,\d x\notag\\
 &\quad + \int_{\mathcal{O}_\lambda\cap\Omega_{R'}} (|u|^{r-1}+|g|^{r-1}) |(u{\rho})_\lambda |{\rho}^{2r-1} \,\d x\notag\\
 &\leq \frac{2rC_2^{p'}\delta_1^{p'}}{p'} \int_{\mathcal{O}^c_\lambda\cap\Omega_{R'}} |\nabla u|^p {\rho}^{(2r-1)p'} \,\d x + \frac{2r\delta_1^{p'}}{p'} \int_{\mathcal{O}^c_\lambda\cap\Omega_{R'}} \beta_2^{p'} {\rho}^{(2r-1)p'} \,\d x\notag\\
 &\quad + \frac{4r}{p\delta_1^p} \int_{\mathcal{O}^c_\lambda\cap\Omega_{R'}} |u|^p |\nabla \rho|^p \,\d x 
 + \frac{1}{p'\delta_2^{p'}} \int_{\mathcal{O}^c_\lambda\cap\Omega_{R'}} |f|^p {\rho}^{2r} \,\d x \notag\\
 &\quad + \frac{\delta_2^p}{p} \int_{\mathcal{O}^c_\lambda\cap\Omega_{R'}} |\nabla u|^p {\rho}^{2r} \,\d x
 + \frac{2r}{p'} \int_{\mathcal{O}^c_\lambda\cap\Omega_{R'}} |f|^p {\rho}^{(2r-1)p'} \,\d x\notag\\
 &\quad + \frac{2r}{p} \int_{\mathcal{O}^c_\lambda\cap\Omega_{R'}} |u|^p |\nabla \rho|^p \,\d x + \frac{1}{r'} \int_{\mathcal{O}^c_\lambda\cap\Omega_{R'}} |g|^r {\rho}^{2r} \,\d x
 + \frac{1}{r} \int_{\mathcal{O}^c_\lambda \cap \Omega_{R'}} |u|^r {\rho}^{2r} \,\d x\notag\\
 &\quad + \int_{\mathcal{O}_\lambda\cap\Omega_{R'}} ( |f|^{p-1}+C_2|\nabla u|^{p-1}+\beta_2 ) |\nabla (u{\rho})_\lambda| {\rho}^{2r-1} \,\d x\notag\\
 &\quad + (2r-1) \int_{\mathcal{O}_\lambda\cap\Omega_{R'}} ( |f|^{p-1}+C_2|\nabla u|^{p-1}+\beta_2 ) {\rho}^{2r-2}|(u{\rho})_\lambda| |\nabla{\rho}| \,\d x\notag\\
 &\quad + \int_{\mathcal{O}_\lambda\cap\Omega_{R'}} (|u|^{r-1}+|g|^{r-1}) |(u{\rho})_\lambda|{\rho}^{2r-1} \,\d x\notag\\%
 &\leq \left( \frac{2rC_2^{p'}\delta_1^{p'}}{p'} + \frac{\delta_2^p}{p} \right) \int_{\mathcal{O}^c_\lambda\cap\Omega_{R'}} |\nabla u|^p {\rho}^{2r} \,\d x\notag\\
 &\quad + \frac{2r}{p}\left( \frac 2 {\delta_1^p} + 1 \right) \int_{\mathcal{O}^c_\lambda\cap\Omega_{R'}} |u|^p |\nabla \rho|^p\,\d x\notag\\
 &\quad +\left( \frac{1}{p'\delta_2^{p'}}+\frac{2r}{p'} \right) \int_{\mathcal{O}^c_\lambda\cap\Omega_{R'}} |f|^p{\rho}^{2r} \,\d x + \frac{2r\delta_1^{p'}}{p'} \int_{\mathcal{O}^c_\lambda\cap\Omega_{R'}} \beta_2^{p'} {\rho}^{(2r-1)p'} \,\d x\notag\\
 &\quad +\frac{1}{r'} \int_{\mathcal{O}^c_\lambda\cap\Omega_{R'}} |g|^r {\rho}^{2r} \,\d x + \frac{1}{r} \int_{\mathcal{O}^c_\lambda\cap\Omega_{R'}} |u|^r {\rho}^{2r} \,\d x\notag\\
 &\quad +\int_{\mathcal{O}_\lambda\cap\Omega_{R'}} ( |f|^{p-1}+C_2|\nabla u|^{p-1}+\beta_2 ) |\nabla (u{\rho})_\lambda| {\rho}^{2r-1} \,\d x\notag\\
 &\quad + (2r-1) \int_{\mathcal{O}_\lambda\cap\Omega_{R'}} ( |f|^{p-1}+C_2|\nabla u|^{p-1}+\beta_2 ) {\rho}^{2r-2} |(u{\rho})_\lambda| |\nabla{\rho}| \,\d x\notag\\
 &\quad + \int_{\mathcal{O}_\lambda\cap\Omega_{R'}} (|u|^{r-1}+|g|^{r-1}) |(u{\rho})_\lambda| {\rho}^{2r-1} \,\d x\notag\\%
 &=:I_1+I_2+I_3+I_4+I_5+I_6+J_1+J_2+J_3.
\end{align}

We can estimate $J_1$ as in the proof of~\cite[Proposition 4.1]{BuSch16},  but we also give a detail for the completeness. Set $G := |f|^{p-1} + C_2|\nabla u|^{p-1}+\beta_2$. For each $\alpha \in (0,1)$ fixed, using Proposition \ref{prop:CubeSplit} and noting that $\sum_{j \in A_i} \psi_j \equiv 1$ on $Q_i$, one can verify that
\begin{align*}
 \MoveEqLeft{
 \int_{\mathcal{O}_\lambda\cap\Omega_{R'}} |G| |\nabla(u{\rho})_\lambda| {\rho}^{2r-1} \,\d x
 }\\
 &= \sum_i \int_{Q_i} |G| \left| \sum_{j\in A_i}\overline{(u{\rho})}_j\nabla\psi_j \right| {\rho}^{2r-1} \,\d x\\
 &= \sum_i\int_{Q_i} |G| \left| \sum_{j\in A_i} \left\{ \overline{(u{\rho})}_j - \overline{(u{\rho})}_i \right\} \nabla\psi_j \right| {\rho}^{2r-1} \,\d x\\
 &\stackrel{(**)}\leq C \sum_i \sum_{j\in A_i} \int_{Q_i} |G| \left( \fint_{(3/2)Q_j}|\nabla (u{\rho})|\,\d x + \fint_{(3/2)Q_i} |\nabla (u{\rho})| \,\d x \right)  {\rho}^{2r-1} \,\d x\\
 &= C \sum_i \sum_{j\in A_i} |Q_i| \left( \fint_{Q_i}|G|{\rho}^{2r-1}\,\d x \right)\\
 &\qquad \times \left( \fint_{(3/2)Q_j} |\nabla (u{\rho})| \,\d x + \fint_{(3/2)Q_i} |\nabla (u{\rho})| \,\d x \right)\\
\prf{ &= C \sum_i \sum_{j\in A_i}|Q_i| \left( \fint_{Q_i} \frac{|G| {\rho}^{2r-1}}{(Mh_\delta)^{\alpha/p}} (Mh_\delta)^{\alpha/p} \,\d x \right)\\
 &\qquad \times \left( \fint_{(3/2)Q_j} \frac{|\nabla (u{\rho})|}{(Mh_\delta)^{\alpha/p'}}(Mh_\delta)^{\alpha/p'} \,\d x + \fint_{(3/2)Q_i}\frac{|\nabla (u{\rho})|}{(Mh_\delta)^{\alpha/p'}}(Mh_\delta)^{\alpha/p'} \,\d x \right)\\}
 &\leq C \sum_i \sum_{j\in A_i} |Q_i| \left( \fint_{Q_i} \frac{|G|^{p'}{\rho}^{(2r-1)p'}}{(Mh_\delta)^{\alpha p'/p}} \,\d x \right)^{1/p'} \left( \fint_{6Q_i}(Mh_\delta)^{\alpha} \,\d x \right)^{1/p}\\
 &\qquad \times \left\{ \left( \fint_{(3/2)Q_j} \frac{|\nabla (u{\rho})|^p}{(Mh_\delta)^{\alpha p/p'}} \,\d x \right)^{1/p} + \left( \fint_{(3/2)Q_i} \frac{|\nabla (u{\rho})|^p}{(Mh_\delta)^{\alpha p/p'}} \,\d x \right)^{1/p} \right\}\\
 &\qquad \times \left( \fint_{6Q_i}(Mh_\delta)^{\alpha } \,\d x \right)^{1/p'}\\
 &\leq C \sum_i \sum_{j\in A_i} |Q_i| \left( \fint_{Q_i}\frac{|G|^{p'}{\rho}^{(2r-1)p'}}{(Mh_\delta)^{\alpha p'/p}} \,\d x \right)^{1/p'} \left( \fint_{6Q_i} (Mh_\delta)^{\alpha } \,\d x \right)\\
 &\qquad \times \left( \fint_{(3/2)Q_j} \frac{|\nabla (u{\rho})|^p}{(Mh_\delta)^{\alpha p/p'}} \,\d x + \fint_{(3/2)Q_i} \frac{|\nabla (u{\rho})|^p}{(Mh_\delta)^{\alpha p/p'}} \,\d x \right)^{1/p}.
\end{align*}
Here we employed~\cite[Lemma 3.1]{BuSch16} along with (vi) of Proposition \ref{prop:CubeSplit}, in particular, $\mathrm{diam}\,(Q_i) |\nabla \psi_i| \leq c(d)$, in order to derive the first inequality $(**)$. Moreover, we also used the fact that $(3/2)Q_j \subset 6Q_i$ if $j \in A_i$. Now, from (ii) of Proposition \ref{prop:CubeSplit}, we find that $9Q_i \cap \mathcal{O}_\lambda^c\neq\emptyset$, and hence, for a.e.~$x_0 \in 9Q_i\cap\mathcal{O}_\lambda^c$, it follows from the definition of $\mathcal{O}_\lambda$ that
\begin{align}\label{ineq:techofBSProp4.1}
 \fint_{6Q_i} (Mh_\delta)^{\alpha } \,\d x 
 &\leq \left(\frac32\right)^d \fint_{9Q_i} (Mh_\delta)^{\alpha} \,\d x \notag\\
 &\leq CM[(Mh_\delta)^\alpha](x_0) 
 \stackrel{(*)}\leq CA(Mh_\delta)^\alpha(x_0)
 \leq CA\lambda^\alpha.
\end{align}
Here we used the fact that $(Mh_\delta)^\alpha \in \mathcal{A}_1$ to derive the inequality $(*)$ of the above (see Lemma \ref{lem:Muckenhoupt} and Definition \ref{def:Muckenhoupt}). Hence it yields
\begin{align*}
 \MoveEqLeft{
 \int_{\mathcal{O}_\lambda\cap\Omega_{R'}} |G| |\nabla(u{\rho})_\lambda| {\rho}^{2r-1} \,\d x
 }\\
 &\leq C \sum_i \sum_{j\in A_i} |Q_i| \lambda^\alpha \left( \fint_{Q_i} \frac{|G|^{p'} {\rho}^{(2r-1)p'}}{(Mh_\delta)^{\alpha p'/p}} \,\d x \right)^{1/p'}\\
 &\qquad \times \left( \fint_{(3/2)Q_j} \frac{|\nabla (u{\rho})|^p}{(Mh_\delta)^{\alpha p/p'}} \,\d x + \fint_{(3/2)Q_i} \frac{|\nabla (u{\rho})|^p}{(Mh_\delta)^{\alpha p/p'}} \,\d x \right)^{1/p}\\
 &= C \sum_i \sum_{j\in A_i} |Q_i| \left( \fint_{Q_i} \frac{\lambda^{\alpha p'/p}|G|^{p'} {\rho}^{(2r-1)p'}}{(Mh_\delta)^{\alpha p'/p}} \,\d x \right)^{1/p'}\\
 &\qquad \times \left( \fint_{(3/2)Q_j} \frac{\lambda^{\alpha p/p'} |\nabla (u{\rho})|^p}{(Mh_\delta)^{\alpha p/p'}} \,\d x + \fint_{(3/2)Q_i} \frac{\lambda^{\alpha p/p'}|\nabla (u{\rho})|^p}{(Mh_\delta)^{\alpha p/p'}} \,\d x \right)^{1/p}\\
 &\leq C \left( \int_{\mathcal{O}_\lambda\cap\Omega_{R'}} \frac{\lambda^{\alpha p'/p}|G|^{p'} {\rho}^{(2r-1)p'}}{(Mh_\delta)^{\alpha p'/p}} \,\d x + \int_{\mathcal{O}_\lambda\cap\Omega_{R'}} \frac{\lambda^{\alpha p/p'}|\nabla (u{\rho})|^p}{(Mh_\delta)^{\alpha p/p'}} \,\d x \right).
\end{align*}
Thus we obtain
\begin{align}\label{J_1}
J_1 &\leq C \left(
 \int_{\mathcal{O}_\lambda\cap\Omega_{R'}} \frac{\lambda^{\alpha p'/p}(|f|^{p}+C_2|\nabla u|^p+\beta_2^{p'}) {\rho}^{2r}}{(Mh_\delta)^{\alpha p'/p}} \, \d x \right.\notag\\
 &\quad \left.+ \int_{\mathcal{O}_\lambda\cap\Omega_{R'}}
 \frac{\lambda^{\alpha p/p'}|\nabla u|^{p}{\rho}^{p}}{(Mh_\delta)^{\alpha p/p'}} \, \d x + \int_{\mathcal{O}_\lambda\cap\Omega_{R'}} \frac{\lambda^{\alpha p/p'}|u|^{p}|\nabla {\rho}|^p}{(Mh_\delta)^{\alpha p/p'}} \,\d x \right).
\end{align}

Let us move on to an estimate for the term $J_2$. Letting $G \in L^{p'}_\omega(\Omega;\mathbb{R}^N)$ and recalling $\supp \psi_j \subset (9/8)Q_j$, we deduce that
\begin{align*}
\MoveEqLeft{
 \int_{\mathcal{O}_\lambda\cap\Omega_{R'}} |G| {\rho}^{2r-2} |(u{\rho})_\lambda| |\nabla{\rho}| \,\d x
}\\
&\leq \sum_i \left( \int_{Q_i} |G| {\rho}^{2r-2} |\nabla{\rho}| \,\d x \right) \left( \sum_{j\in A_i} |\overline{(u{\rho})}_j| \right)\\
&\leq \sum_i |Q_i| \left( \fint_{Q_i} |G| {\rho}^{2r-2} |\nabla{\rho}| \,\d x \right) \left( \sum_{j\in A_i} \fint_{(9/8)Q_j} |u| {\rho} \,\d x \right)\\
 &\leq \sum_i \sum_{j\in A_i} |Q_i| \left( \frac{r'}{p'} \fint_{Q_i} |G|^{p'/r'} {\rho}^{(2r-2)p'/r'} \,\d x \right.\\
 &\quad \left. + \frac{p'-r'}{p'} \fint_{Q_i} |\nabla{\rho}|^{p'/(p'-r')} \,\d x \right) \left( \fint_{(9/8)Q_j} |u| {\rho} \,\d x \right)\\
 &= \sum_i \sum_{j\in A_i} |Q_i| \left( \frac{r'}{p'} \fint_{Q_i} \frac{|G|^{p'/r'} {\rho}^{(2r-2)p'/r'}}{(Mh_\delta)^{\alpha /r}}(Mh_\delta)^{\alpha /r} \,\d x \right.\\
 &\quad + \left. \frac{p'-r'}{p'} \fint_{Q_i} \frac{|\nabla{\rho}|^{p'/(p'-r')}}{(Mh_\delta)^{\alpha /r}}(Mh_\delta)^{\alpha /r} \,\d x \right)\\
 &\quad \times \left( \fint_{(9/8)Q_j} \frac{|u|{\rho}}{(Mh_\delta)^{\alpha  /r'}}(Mh_\delta)^{\alpha /r'} \,\d x \right).
\end{align*}
Since $(2r-2)p' > 2r$, we have
\begin{align*}
 \MoveEqLeft{
\fint_{Q_i} \frac{|G|^{p'/r'}{\rho}^{(2r-2)p'/r'}}{(Mh_\delta)^{\alpha /r}}(Mh_\delta)^{\alpha /r} \,\d x
 }\\
 &\leq \left( \fint_{Q_i} \frac{|G|^{p'}{\rho}^{(2r-2)p'}}{(Mh_\delta)^{\alpha r'/r}} \,\d x \right)^{1/r'} \left( \fint_{Q_i}(Mh_\delta)^{\alpha} \,\d x\right)^{1/r}\\
&\leq 6^{d/r} \left( \fint_{Q_i} \frac{|G|^{p'}{\rho}^{(2r-2)p'}}{(Mh_\delta)^{\alpha r'/r}}\d x \right)^{1/r'} \left( \fint_{6Q_i}(Mh_\delta)^{\alpha} \,\d x \right)^{1/r}\\
&\leq 6^{d/r} \left( \fint_{Q_i} \frac{|G|^{p'}{\rho}^{2r}}{(Mh_\delta)^{\alpha r'/r}} \,\d x \right)^{1/r'} \left(\fint_{6Q_i}(Mh_\delta)^{\alpha} \,\d x \right)^{1/r}.
\end{align*}
Similarly, we also see that
\begin{align*}
 \MoveEqLeft{
 \fint_{Q_i} \frac{|\nabla{\rho}|^{p'/(p'-r')}}{(Mh_\delta)^{\alpha /r}}(Mh_\delta)^{\alpha /r} \,\d x
 }\\
\prf{ &\leq \left( \fint_{Q_i} \frac{|\nabla{\rho}|^{p'r'/(p'-r')}}{(Mh_\delta)^{\alpha r'/r}} \,\d x \right)^{1/r'} \left( \fint_{Q_i}(Mh_\delta)^{\alpha} \,\d x \right)^{1/r}\\}
 &\leq 6^{d/r} \left( \fint_{Q_i} \frac{|\nabla{\rho}|^{p'r'/(p'-r')}}{(Mh_\delta)^{\alpha r'/r}} \,\d x \right)^{1/r'} \left( \fint_{6Q_i} (Mh_\delta)^{\alpha} \,\d x \right)^{1/r}.
\end{align*}
Moreover, by virtue of (iii) of Proposition \ref{prop:CubeSplit}, we find that $(9/8)Q_j \subset 6Q_i$ if $j \in A_i$. Therefore for any $F \in L^1_{{\loc}}(\mathbb{R}^d)$, we have
$$
\fint_{(9/8)Q_j} |F| \,\d x
\leq \left( \frac{6}{9/8} \right)^d 2^d \fint_{6Q_i} |F| \,\d x \quad \mbox{ if } \ j \in A_i.
$$
From this observation, we infer that
\begin{align*}
\MoveEqLeft{
 \fint_{(9/8)Q_j} \frac{|u|{\rho}}{(Mh_\delta)^{\alpha/r'}}(Mh_\delta)^{\alpha/r'} \,\d x
 }\\
 &\leq \left( \fint_{(9/8)Q_j} \frac{|u|^r{\rho}^{r}}{(Mh_\delta)^{\alpha r/r'}} \,\d x \right)^{1/r} \left( \fint_{(9/8)Q_j}(Mh_\delta)^{\alpha} \,\d x \right)^{1/r'}\\
 &\leq C \left( \fint_{(9/8)Q_j} \frac{|u|^r{\rho}^{r}}{(Mh_\delta)^{\alpha r/r'}} \,\d x \right)^{1/r} \left( \fint_{6Q_i} (Mh_\delta)^{\alpha} \,\d x \right)^{1/r'}.
\end{align*}
Here, using (iv) and (v) of Proposition \ref{prop:CubeSplit}, we also remark that
$$
\sum_i \sum_{j\in A_i} |Q_i| \fint_{Q_i} |F|\chi_{\Omega_{R'}} \,\d x
\leq 4^d \sum_i \int_{Q_i} |F|\chi_{\Omega_{R'}} \,\d x = 4^d \int_{\mathcal{O}_\lambda\cap\Omega_{R'}} |F| \,\d x
$$
for $F \in L^1_{\loc}(\mathbb{R}^d)$. Furthermore, 
as in deriving \eqref{poco}, we also note that
\begin{align*}
\sum_i \sum_{j\in A_i} |Q_i| \fint_{(9/8)Q_j} |F|\chi_{\Omega_{R'}} \,\d x
&\leq \left( 2 \cdot \frac{8}{9} \right)^d \sum_i\sum_{j\in A_i} \int_{(9/8)Q_j} |F|\chi_{\Omega_{R'}} \,\d x\\
&\leq 4^d \left(\frac{16}{9}\right)^d \sum_j \int_{(9/8)Q_j} |F|\chi_{\Omega_{R'}} \,\d x\\
&\leq 4^{2d} \left(\frac{16}{9}\right)^d \int_{\mathcal{O}_\lambda\cap\Omega_{R'}} |F| \,\d x.
\end{align*}
Thus one obtains
\begin{align*}
 \MoveEqLeft{
 \int_{\mathcal{O}_\lambda\cap\Omega_{R'}} |G| {\rho}^{2r-2} |(u{\rho})_\lambda| |\nabla{\rho}| \,\d x
 }\\
 \prf{&\leq \sum_i \sum_{j\in A_i} |Q_i| \left( \frac{r'}{p'} \fint_{Q_i} \frac{|G|^{p'/r'} {\rho}^{(2r-2)p'/r'}}{(Mh_\delta)^{\alpha /r}} (Mh_\delta)^{\alpha /r} \,\d x \right.\\
 &\quad + \left. \frac{p'-r'}{p'} \fint_{Q_i} \frac{|\nabla{\rho}|^{p'/(p'-r')}}{(Mh_\delta)^{\alpha /r}}(Mh_\delta)^{\alpha /r} \,\d x \right)\\
 &\qquad\times \left(\fint_{(9/8)Q_j} \frac{|u|{\rho}}{(Mh_\delta)^{\alpha /r'}}(Mh_\delta)^{\alpha /r'} \,\d x\right)\\}
 &\leq C\sum_i\sum_{j\in A_i} |Q_i| \left(\fint_{Q_i} \frac{|G|^{p'}{\rho}^{2r}}{(Mh_\delta)^{\alpha r'/r}} \,\d x\right)^{1/r'}\\
 &\qquad\times \left(\fint_{(9/8)Q_j} \frac{|u|^r{\rho}^{r}}{(Mh_\delta)^{\alpha r/r'}}\,\d x\right)^{1/r} \left(\fint_{6Q_i}(Mh_\delta)^{\alpha}\,\d x\right)\\
 &\quad +C\sum_i\sum_{j\in A_i}|Q_i|\left(\fint_{Q_i} \frac{|\nabla{\rho}|^{p'r'/(p'-r')}}{(Mh_\delta)^{\alpha r'/r}}\,\d x\right)^{1/r'}\\
 &\qquad\times\left(\fint_{(9/8)Q_j}\frac{|u|^r{\rho}^{r}}{(Mh_\delta)^{\alpha  r/r'}}\,\d x\right)^{1/r}\left(\fint_{6Q_i}(Mh_\delta)^{\alpha}\,\d x\right)\\
 &\stackrel{\eqref{ineq:techofBSProp4.1}}\leq \prf{C\lambda^\alpha\sum_i\sum_{j\in A_i}|Q_i|\left(\fint_{Q_i} \frac{|G|^{p'}{\rho}^{2r}}{(Mh_\delta)^{\alpha r'/r}}\,\d x\right)^{1/r'}\\
 &\qquad\times\left(\fint_{(9/8)Q_j}\frac{|u|^r\rho^{r}}{(Mh_\delta)^{\alpha r/r'}}\,\d x\right)^{1/r}\\
 &\quad +C\lambda^\alpha\sum_i\sum_{j\in A_i}|Q_i|\left(\fint_{Q_i} \frac{|\nabla{\rho}|^{p'r'/(p'-r')}}{(Mh_\delta)^{\alpha r'/r}}\,\d x\right)^{1/r'}\\
 &\qquad\times\left(\fint_{(9/8)Q_j}\frac{|u|^r\rho^{r}}{(Mh_\delta)^{\alpha  r/r'}}\,\d x\right)^{1/r}\\
 &=} C\sum_i\sum_{j\in A_i}|Q_i|\left(\fint_{Q_i}\frac{\lambda^{\alpha r'/r}|G|^{p'}{\rho}^{2r}}{(Mh_\delta)^{\alpha r'/r}}\,\d x\right)^{1/r'}\\
 &\qquad\times\left(\fint_{(9/8)Q_j}\frac{\lambda^{\alpha r/r'}|u|^r\rho^{r}}{(Mh_\delta)^{\alpha r/r'}}\d x\right)^{1/r}\\
 &\quad +C\sum_i\sum_{j\in A_i}|Q_i|\left(\fint_{Q_i}\frac{\lambda^{\alpha r'/r}|\nabla{\rho}|^{p'r'/(p'-r')}}{(Mh_\delta)^{\alpha r'/r}}\,\d x\right)^{1/r'}\\
 &\qquad\times\left(\fint_{(9/8)Q_j}\frac{\lambda^{\alpha r/r'}|u|^r\rho^{r}}{(Mh_\delta)^{\alpha  r/r'}}\,\d x\right)^{1/r}\\
  &\leq C\sum_i\sum_{j\in A_i}|Q_i|\left\{\fint_{Q_i}\frac{\lambda^{\alpha r'/r}|G|^{p'}{\rho}^{2r}}{(Mh_\delta)^{\alpha r'/r}}\,\d x+\fint_{(9/8)Q_j}\frac{\lambda^{\alpha r/r'}|u|^r\rho^{r}}{(Mh_\delta)^{\alpha r/r'}}\,\d x\right\}\\
 &\quad +C\sum_i\sum_{j\in A_i}|Q_i| \left\{\fint_{Q_i}\frac{\lambda^{\alpha r'/r}|\nabla{\rho}|^{p'r'/(p'-r')}}{(Mh_\delta)^{\alpha r'/r}}\,\d x\right.\\
 &\qquad\left.+\fint_{(9/8)Q_j}\frac{\lambda^{\alpha r/r'}|u|^r\rho^{r}}{(Mh_\delta)^{\alpha  r/r'}}\,\d x\right\}\\
 &\leq C\left(\int_{\mathcal{O}_\lambda\cap\Omega_{R'}}\frac{\lambda^{\alpha r'/r}|G|^{p'}{\rho}^{2r}}{(Mh_\delta)^{\alpha r'/r}}\,\d x+\int_{\mathcal{O}_\lambda\cap\Omega_{R'}}\frac{\lambda^{\alpha r/r'}|u|^r\rho^{r}}{(Mh_\delta)^{\alpha r/r'}}\,\d x\right.\\
 &\quad \left.+\int_{\mathcal{O}_\lambda\cap\Omega_{R'}}\frac{\lambda^{\alpha r'/r}|\nabla{\rho}|^{p'r'/(p'-r')}}{(Mh_\delta)^{\alpha r'/r}}\,\d x\right).
\end{align*}
Hence setting $G = (2r-1)(|f|^{p-1}+C_2|\nabla u|^{p-1}+\beta_2)$, we observe that
\begin{align}\label{J_2}
 J_2&=(2r-1)\int_{\mathcal{O}_\lambda\cap\Omega_{R'}}(|f|^{p-1}+C_2|\nabla u|^{p-1}+\beta_2){\rho}^{2r-2}|(u{\rho})_\lambda||\nabla{\rho}|\,\d x\notag\\
 &\leq C \left( \int_{\mathcal{O}_\lambda\cap\Omega_{R'}}\frac{\lambda^{\alpha r'/r}(|f|^p+C_2|\nabla u|^p+\beta_2^{p'}){\rho}^{2r}}{(Mh_\delta)^{\alpha r'/r}}\,\d x \right.\notag\\
 &\quad + \left. \int_{\mathcal{O}_\lambda\cap\Omega_{R'}}\frac{\lambda^{\alpha r/r'}|u|^r\rho^{r}}{(Mh_\delta)^{\alpha r/r'}}\,\d x + \int_{\mathcal{O}_\lambda\cap\Omega_{R'}}\frac{\lambda^{\alpha r'/r}|\nabla{\rho}|^{p'r'/(p'-r')}}{(Mh_\delta)^{\alpha r'/r}}\,\d x\right).
\end{align}

Finally, we estimate $J_3$. Noting that $0\leq{\rho}\leq 1$ in $\mathbb{R}^d$, we observe that
\begin{align*}
J_3&=\int_{\mathcal{O}_\lambda\cap\Omega_{R'}}(|u|^{r-1}+|g|^{r-1})|(u{\rho})_\lambda|{\rho}^{2r-1}\,\d x\\
&\leq\sum_i\int_{Q_i}(|u|^{r-1}+|g|^{r-1})\left(\sum_{j\in A_i}|\overline{(u{\rho})}_j|\right){\rho}^{2r-1}\,\d x\\
&=\sum_i\sum_{j\in A_i}|Q_i|\left(\fint_{Q_i}(|u|^{r-1}+|g|^{r-1}){\rho}^{2r-1}\,\d x\right)\left(\fint_{(9/8)Q_j}|u|{\rho}\,\d x\right)\\
&=\sum_i\sum_{j\in A_i}|Q_i|\left(\fint_{Q_i}\frac{(|u|^{r-1}+|g|^{r-1}){\rho}^{2r-1}}{(Mh_\delta)^{\alpha/r}}(Mh_\delta)^{\alpha/r}\,\d x\right)\\
&\qquad\times\left(\fint_{(9/8)Q_j}\frac{|u|{\rho}}{(Mh_\delta)^{\alpha/r'}}(Mh_\delta)^{\alpha/r'}\,\d x\right).
\end{align*}
Here we can similarly derive that
\begin{align*}
\MoveEqLeft{
\fint_{Q_i}\frac{(|u|^{r-1}+|g|^{r-1}){\rho}^{2r-1}}{(Mh_\delta)^{\alpha/r}}(Mh_\delta)^{\alpha/r}\,\d x
}\\
 &\leq C\left(\fint_{Q_i}\frac{(|u|^{r}+|g|^r)\rho^{2r}}{(Mh_\delta)^{\alpha r'/r}}\,\d x\right)^{1/r'}\left(\fint_{6Q_i}(Mh_\delta)^{\alpha}\,\d x\right)^{1/r}
\end{align*}
and
\begin{align*}
 \MoveEqLeft{
 \fint_{(9/8)Q_j}\frac{|u|{\rho}}{(Mh_\delta)^{\alpha/r'}}(Mh_\delta)^{\alpha/r'}\,\d x
 }\\
 &\leq C\left(\fint_{(9/8)Q_j}\frac{|u|^r{\rho}^{r}}{(Mh_\delta)^{\alpha r/r'}}\,\d x\right)^{1/r}\left(\fint_{6Q_i}(Mh_\delta)^{\alpha}\,\d x\right)^{1/r'}.
\end{align*}
Thus it follows that
\begin{align}\label{J_3}
 J_3\prf{&\leq\sum_i\sum_{j\in A_i}|Q_i|\left(\fint_{Q_i}\frac{(|u|^{r-1}+|g|^{r-1}){\rho}^{2r-1}}{(Mh_\delta)^{\alpha/r}}(Mh_\delta)^{\alpha/r}\,\d x\right)\notag\\
 &\qquad\times\left(\fint_{(9/8)Q_j}\frac{|u|{\rho}}{(Mh_\delta)^{\alpha/r'}}(Mh_\delta)^{\alpha/r'}\,\d x\right)\notag\\
 &\leq C\sum_i\sum_{j\in A_i}|Q_i|\left(\fint_{Q_i}\frac{(|u|^{r}+|g|^r)\rho^{2r}}{(Mh_\delta)^{\alpha r'/r}}\,\d x\right)^{1/r'}\notag\\
 &\qquad\times\left(\fint_{(9/8)Q_j}\frac{|u|^r{\rho}^{r}}{(Mh_\delta)^{\alpha r/r'}}\,\d x\right)^{1/r}\left(\fint_{6Q_i}(Mh_\delta)^{\alpha}\,\d x\right)\notag\\}
 &\stackrel{\eqref{ineq:techofBSProp4.1}}\leq C\lambda^\alpha\sum_i\sum_{j\in A_i}|Q_i|\left(\fint_{Q_i}\frac{(|u|^{r}+|g|^r)\rho^{2r}}{(Mh_\delta)^{\alpha r'/r}}\,\d x\right)^{1/r'}\left(\fint_{(9/8)Q_j}\frac{|u|^r{\rho}^{r}}{(Mh_\delta)^{\alpha r/r'}}\,\d x\right)^{1/r}\notag\\
\prf{ &= C\sum_i\sum_{j\in A_i}|Q_i|\left(\fint_{Q_i}\frac{\lambda^{\alpha r'/r}(|u|^{r}+|g|^r)\rho^{2r}}{(Mh_\delta)^{\alpha r'/r}}\,\d x\right)^{1/r'}\left(\fint_{(9/8)Q_j}\frac{\lambda^{\alpha r/r'}|u|^r{\rho}^{r}}{(Mh_\delta)^{\alpha r/r'}}\,\d x\right)^{1/r}\notag\\}
 &\leq C\sum_i\sum_{j\in A_i}|Q_i|\left(\fint_{Q_i}\frac{\lambda^{\alpha r'/r}(|u|^{r}+|g|^r)\rho^{2r}}{(Mh_\delta)^{\alpha r'/r}}\,\d x+\fint_{(9/8)Q_j}\frac{\lambda^{\alpha r/r'}|u|^r{\rho}^{r}}{(Mh_\delta)^{\alpha r/r'}}\,\d x\right)\notag\\
 &\leq C\left(\int_{\mathcal{O}_\lambda\cap\Omega_{R'}}\frac{\lambda^{\alpha r'/r}(|u|^{r}+|g|^r)\rho^{2r}}{(Mh_\delta)^{\alpha r'/r}}\,\d x+\int_{\mathcal{O}_\lambda\cap\Omega_{R'}}\frac{\lambda^{\alpha r/r'}|u|^r{\rho}^{r}}{(Mh_\delta)^{\alpha r/r'}}\,\d x\right),
\end{align}
where the last inequality is derived as in \eqref{poco}.

Combining \eqref{int I,J}, \eqref{J_1}, \eqref{J_2} and \eqref{J_3}, we infer that
\prf{\begin{align*}
 \MoveEqLeft{
 \int_{\mathcal{O}^c_\lambda\cap\Omega_{R'}}(C_1|\nabla u|^{p}-\beta_1)\rho^{2r}\,\d x+\frac{1}{r'}\int_{\mathcal{O}^c_\lambda\cap\Omega_{R'}}|u|^{r}\rho^{2r}\,\d x
 }\\
 &\leq I_1+I_2+I_3+I_4+I_5+J_1+J_2+J_3\\
 &\leq\left(\frac{2rC_2\delta_1^{p'}}{p'}+\frac{\delta_2^p}{p}\right)\int_{\mathcal{O}^c_\lambda\cap\Omega_{R'}}|\nabla u|^p{\rho}^{2r}\,\d x+\frac{2r}{p}\left(\frac{2}{\delta_1^p}+1\right)\int_{\mathcal{O}^c_\lambda\cap\Omega_{R'}}|u|^p|\nabla \rho|^p\,\d x\\
 &\quad + \left( \frac{1}{p'\delta_2^{p'}}+\frac{2r}{p'}\right)\int_{\mathcal{O}^c_\lambda\cap\Omega_{R'}}|f|^p{\rho}^{2r}\,\d x+\frac{2r\delta_1^{p'}}{p'}\int_{\mathcal{O}^c_\lambda\cap\Omega_{R'}}\beta_2^{p'}{\rho}^{(2r-1)p'}\,\d x\\
 &\quad +\frac{1}{r'}\int_{\mathcal{O}^c_\lambda\cap\Omega_{R'}}|g|^{r}\rho^{2r}\,\d x\\
 &\quad +C \left(\int_{\mathcal{O}_\lambda\cap\Omega_{R'}}\frac{\lambda^{\alpha p'/p}(|f|^{p}+C_2|\nabla u|^p+\beta_2^{p'}){\rho}^{2r}}{(Mh_\delta)^{\alpha p'/p}}\,\d x\right.\notag\\
 &\qquad\left.+\int_{\mathcal{O}_\lambda\cap\Omega_{R'}}\frac{\lambda^{\alpha p/p'}|\nabla u|^{p}{\rho}^{p}}{(Mh_\delta)^{\alpha p/p'}}\,\d x + \int_{\mathcal{O}_\lambda\cap\Omega_{R'}}\frac{\lambda^{\alpha p/p'}|u|^{p}|\nabla {\rho}|^p}{(Mh_\delta)^{\alpha p/p'}}\,\d x\right.\\
 &\qquad +\int_{\mathcal{O}_\lambda\cap\Omega_{R'}}\frac{\lambda^{\alpha r'/r}(|f|^p+C_2|\nabla u|^p+\beta_2^{p'}){\rho}^{2r}}{(Mh_\delta)^{\alpha r'/r}}\,\d x\\
 &\qquad +\int_{\mathcal{O}_\lambda\cap\Omega_{R'}}\frac{\lambda^{\alpha r/r'}|u|^r\rho^{r}}{(Mh_\delta)^{\alpha r/r'}}\,\d x
 +\int_{\mathcal{O}_\lambda\cap\Omega_{R'}}\frac{\lambda^{\alpha r'/r}|\nabla{\rho}|^{p'r'/(p'-r')}}{(Mh_\delta)^{\alpha r'/r}}\,\d x\\
 &\qquad +\left.\int_{\mathcal{O}_\lambda\cap\Omega_{R'}}\frac{\lambda^{\alpha r'/r}(|u|^{r}+|g|^r)\rho^{2r}}{(Mh_\delta)^{\alpha r'/r}}\,\d x+\int_{\mathcal{O}_\lambda\cap\Omega_{R'}}\frac{\lambda^{\alpha r/r'}|u|^r{\rho}^{r}}{(Mh_\delta)^{\alpha r/r'}}\,\d x\right).
\end{align*}
Therefore we have}
\begin{align}
 \MoveEqLeft{\left[C_1-\left(\frac{2rC_2\delta_1^{p'}}{p'}+\frac{\delta_2^p}{p}\right)\right]\int_{\mathcal{O}^c_\lambda\cap\Omega_{R'}}|\nabla u|^{p}\rho^{2r}\,\d x+\frac{1}{r'}\int_{\mathcal{O}^c_\lambda\cap\Omega_{R'}}|u|^{r}\rho^{2r}\,\d x}\notag\\
 &\leq\int_{\mathcal{O}^c_\lambda\cap\Omega_{R'}}\beta_1\rho^{2r}\,\d x+\frac{2r}{p}\left(\frac{2}{\delta_1^p}+1\right)\int_{\mathcal{O}^c_\lambda\cap\Omega_{R'}}|u|^p |\nabla \rho|^p\,\d x\notag\\
 &\quad +\left(\frac{1}{p'\delta_2^{p'}}+\frac{2r}{p'}\right)\int_{\mathcal{O}^c_\lambda\cap\Omega_{R'}}|f|^{p}{\rho}^{2r}\,\d x\notag\\
 &\quad +\frac{2r\delta_1^{p'}}{p'}\int_{\mathcal{O}^c_\lambda\cap\Omega_{R'}}\beta_2^{p'}{\rho}^{(2r-1)p'}\,\d x+\frac{1}{r'}\int_{\mathcal{O}^c_\lambda\cap\Omega_{R'}}|g|^{r}\rho^{2r}\,\d x\notag\\
 &\quad +C\left(\int_{\mathcal{O}_\lambda\cap\Omega_{R'}}\frac{\lambda^{\alpha p'/p}(|f|^{p}+C_2|\nabla u|^p+\beta_2^{p'})\rho^{2r}}{(Mh_\delta)^{\alpha p'/p}}\,\d x \right.\notag\\
 &\qquad +\int_{\mathcal{O}_\lambda\cap\Omega_{R'}}\frac{\lambda^{\alpha p/p'}|\nabla u|^p{\rho}^{p}}{(Mh_\delta)^{\alpha p/p'}}\,\d x
 +\int_{\mathcal{O}_\lambda\cap\Omega_{R'}}\frac{\lambda^{\alpha p/p'}|u|^{p}|\nabla {\rho}|^p}{(Mh_\delta)^{\alpha p/p'}}\,\d x\notag\\
 &\qquad +\int_{\mathcal{O}_\lambda\cap\Omega_{R'}}\frac{\lambda^{\alpha r'/r}(|f|^p+C_2|\nabla u|^p+\beta_2^{p'}){\rho}^{2r}}{(Mh_\delta)^{\alpha r'/r}}\,\d x\notag\\
 &\qquad +\int_{\mathcal{O}_\lambda\cap\Omega_{R'}}\frac{\lambda^{\alpha r/r'}|u|^r\rho^{r}}{(Mh_\delta)^{\alpha r/r'}}\,\d x 
 +\int_{\mathcal{O}_\lambda\cap\Omega_{R'}}\frac{\lambda^{\alpha r'/r}|\nabla{\rho}|^{p'r'/(p'-r')}}{(Mh_\delta)^{\alpha r'/r}}\,\d x\notag\\
 &\qquad \left. +\int_{\mathcal{O}_\lambda\cap\Omega_{R'}}\frac{\lambda^{\alpha r'/r}(|u|^{r}+|g|^r)\rho^{2r}}{(Mh_\delta)^{\alpha r'/r}}\,\d x
 +\int_{\mathcal{O}_\lambda\cap\Omega_{R'}}\frac{\lambda^{\alpha r'/r}|u|^{r}\rho^{r}}{(Mh_\delta)^{\alpha r'/r}}\,\d x\right).\label{ququ}
\end{align}

Now, let $F \in L^1_{{\loc}}(\mathbb{R}^d)$ and $\vep > 0$. Using Fubini's theorem, we note that
\begin{align*}
 \int_0^\infty\lambda^{-(1+\varepsilon)}\left(\int_{\mathcal{O}_\lambda^c\cap\Omega_{R'}}|F|\,\d x\right)\,\d\lambda
 &=\int_{\Omega_{R'}}|F|\left(\int_{(Mh_\delta)(x)}^\infty\lambda^{-(1+\varepsilon)}\,\d\lambda\right)\,\d x\\
 &=\frac{1}{\varepsilon}\int_{\Omega_{R'}}\frac{|F|}{(Mh_\delta)^\varepsilon}\,\d x.
\end{align*}
Similarly, taking $\beta>0$ so that $\alpha\beta-\varepsilon>0$, we find that
\begin{align*}
\MoveEqLeft{
 \int_0^\infty\lambda^{-(1+\varepsilon)}\left(\int_{\mathcal{O}_\lambda\cap\Omega_{R'}}\lambda^{\alpha\beta}\frac{|F|}{(Mh_\delta)^{\alpha\beta}}\,\d x\right)\,\d\lambda
 }\\
 &=\int_{\Omega_{R'}}\frac{|F|}{(Mh_\delta)^{\alpha\beta}}\left(\int_0^{(Mh_\delta)(x)}\lambda^{\alpha\beta-(1+\varepsilon)}\,\d\lambda\right)\,\d x\\
 &=\frac{1}{\alpha\beta-\varepsilon}\int_{\Omega_{R'}}\frac{|F|}{(Mh_\delta)^{\varepsilon}}\,\d x.
\end{align*}
Set $\overline{(p-1)} := \min\{p-1,\ (p-1)^{-1} \} = \min\{p-1,\ p'-1 \} >0$ and $\overline{(r-1)} := \min\{r-1,\ (r-1)^{-1} \} = \min\{r-1,\ r'-1 \} >0$. Let $\varepsilon \in (0,\alpha\min\{\overline{(p-1)},\overline{(r-1)}\})$ be arbitrarily fixed. Then 
$$
\alpha \min \{ r'/r, \ r/r', \ p'/p, \ p/p'\} - \vep > 0.
$$
Divide both sides of \eqref{ququ} by $\lambda^{1+\varepsilon}$ and integrate it over $(0,\infty)$ in $\lambda$. From the above observations, we deduce that
\begin{align*}
\lefteqn{
 \frac 1 \vep \left[ C_1-\left(\frac{2rC_2\delta_1^{p'}}{p'}+\frac{\delta_2^p}{p}\right)\right]\int_{\Omega_{R'}}\frac{|\nabla u|^{p}\rho^{2r}}{(Mh_\delta)^\varepsilon}\,\d x+\frac{1}{\varepsilon r'}\int_{\Omega_{R'}}\frac{|u|^{r}\rho^{2r}}{(Mh_\delta)^\varepsilon}\,\d x
}\\
 &\leq\frac{1}{\varepsilon}\int_{\Omega_{R'}}\frac{\beta_1\rho^{2r}}{(Mh_\delta)^\varepsilon}\,\d x + \frac{2r}{\vep p}\left(\frac{2}{\delta_1^p}+1\right)\int_{\Omega_{R'}}\frac{|u|^p |\nabla \rho|^p}{(Mh_\delta)^\varepsilon}\,\d x\\
 &\quad +\frac{1}{\varepsilon}\left(\frac{1}{p'\delta_2^{p'}}+\frac{2r}{p'}\right)\int_{\Omega_{R'}}\frac{|f|^{p}{\rho}^{2r}}{(Mh_\delta)^\varepsilon}\,\d x\\
 &\quad +\frac{2r\delta_1^{p'}}{\vep p'}\int_{\Omega_{R'}}\frac{\beta_2^{p'}{\rho}^{(2r-1)p'}}{(Mh_\delta)^\varepsilon}\,\d x+\frac{1}{\varepsilon r'}\int_{\Omega_{R'}}\frac{|g|^r\rho^{2r}}{(Mh_\delta)^\varepsilon}\,\d x\\
 &\quad +\frac{C}{\alpha\overline{(p-1)}-\varepsilon}\left(\int_{\Omega_{R'}}\frac{(|f|^{p}+C_2|\nabla u|^p+\beta_2^{p'})\rho^{2r}}{(Mh_\delta)^{\varepsilon}}\,\d x \right.\\
 &\qquad \left. + \int_{\Omega_{R'}}\frac{|\nabla u|^p{\rho}^{p}}{(Mh_\delta)^{\varepsilon}}\,\d x + \int_{\Omega_{R'}}\frac{|u|^{p}|\nabla {\rho}|^p}{(Mh_\delta)^{\varepsilon}}\,\d x \right) \\
 &\quad + \frac{C}{\alpha\overline{(r-1)}-\varepsilon} \left( \int_{\Omega_{R'}}\frac{(|f|^p+C_2|\nabla u|^p+\beta_2^{p'}){\rho}^{2r}}{(Mh_\delta)^{\varepsilon}}\,\d x \right.\\
 &\qquad + \int_{\Omega_{R'}}\frac{|u|^r\rho^{r}}{(Mh_\delta)^{\varepsilon}}\,\d x + \int_{\Omega_{R'}}\frac{|\nabla{\rho}|^{p'r'/(p'-r')}}{(Mh_\delta)^{\varepsilon}}\,\d x\\
 &\qquad \left. + \int_{\Omega_{R'}}\frac{(|u|^{r}+|g|^r)\rho^{2r}}{(Mh_\delta)^{\varepsilon}}\,\d x + \int_{\Omega_{R'}}\frac{|u|^{r}\rho^{r}}{(Mh_\delta)^{\varepsilon}}\,\d x \right)\\
\prf{ &= \frac{1}{\varepsilon} \int_{\Omega_{R'}}\frac{\beta_1\rho^{2r}}{(Mh_\delta)^\varepsilon}\,\d x + \frac{2r}{\vep p}\left(\frac{2}{\delta_1^p}+1\right)\int_{\Omega_{R'}}\frac{|u|^p|\nabla \rho|^p}{(Mh_\delta)^\varepsilon}\,\d x\\
 &\quad+\frac{1}{\varepsilon}\left(\frac{1}{p'\delta_2^{p'}} + \frac{2r}{p'}\right)\int_{\Omega_{R'}}\frac{|f|^{p}{\rho}^{2r}}{(Mh_\delta)^\varepsilon}\,\d x + \frac{2r\delta_1^{p'}}{\vep p'}\int_{\Omega_{R'}}\frac{\beta_2^{p'}{\rho}^{(2r-1)p'}}{(Mh_\delta)^\varepsilon}\,\d x\\
 &\quad +\left(\frac{C}{\alpha\overline{(p-1)}-\varepsilon}+\frac{C}{\alpha\overline{(r-1)}-\varepsilon}\right)\int_{\Omega_{R'}}\frac{|f|^{p}\rho^{2r}}{(Mh_\delta)^{\varepsilon}}\,\d x\\
 &\quad +\left(\frac{CC_2}{\alpha\overline{(p-1)}-\varepsilon}+\frac{CC_2}{\alpha\overline{(r-1)}-\varepsilon}\right)\int_{\Omega_{R'}}\frac{|\nabla u|^p{\rho}^{2r}}{(Mh_\delta)^{\varepsilon}}\,\d x\\
 &\quad +\frac{C}{\alpha\overline{(p-1)}-\varepsilon}\int_{\Omega_{R'}}\frac{|\nabla u|^p{\rho}^{p}}{(Mh_\delta)^{\varepsilon}}\,\d x\\
 &\quad +\left(\frac{C}{\alpha\overline{(p-1)}-\varepsilon}+\frac{C}{\alpha\overline{(r-1)}-\varepsilon}\right)\int_{\Omega_{R'}}\frac{\beta_2^{p'}{\rho}^{2r}}{(Mh_\delta)^{\varepsilon}}\,\d x\\
 &\quad +\left(\frac{1}{\varepsilon r'}+\frac{C}{\alpha\overline{(r-1)}-\varepsilon}\right)\int_{\Omega_{R'}}\frac{|g|^r\rho^{2r}}{(Mh_\delta)^\varepsilon}\,\d x\\
 &\quad +\frac{C}{\alpha\overline{(p-1)}-\varepsilon}\int_{\Omega_{R'}}\frac{|u|^{p}|\nabla {\rho}|^p}{(Mh_\delta)^{\varepsilon}}\,\d x\\
 &\quad +\frac{C}{\alpha\overline{(r-1)}-\varepsilon}\int_{\Omega_{R'}}\left(\frac{|u|^{r}\rho^{r}}{(Mh_\delta)^{\varepsilon}}+\frac{|u|^{r}\rho^{2r}}{(Mh_\delta)^{\varepsilon}}\right)\,\d x\\
 &\quad +\frac{C}{\alpha\overline{(r-1)}-\varepsilon}\int_{\Omega_{R'}}\frac{|\nabla{\rho}|^{p'r'/(p'-r')}}{(Mh_\delta)^{\varepsilon}}\,\d x\\ }
 &\leq\frac{1}{\varepsilon}\int_{\Omega_{R'}}\frac{\beta_1\rho^{2r}}{(Mh_\delta)^\varepsilon}\,\d x + \left[ \frac{2r}{\vep p}\left(\frac{2}{\delta_1^p}+1\right) + \frac{C}{\alpha\overline{(p-1)}-\varepsilon} \right] \int_{\Omega_{R'}}\frac{|u|^p|\nabla \rho|^p}{(Mh_\delta)^\varepsilon}\,\d x\\
 &\quad +\left[ \frac 1 \vep \left(\frac{1}{p'\delta_2^{p'}}+\frac{2r}{p'}\right)+\frac{C}{\alpha\overline{(p-1)}-\varepsilon}+\frac{C}{\alpha\overline{(r-1)}-\varepsilon}\right]\int_{\Omega_{R'}}\frac{|f|^{p}\rho^{2r}}{(Mh_\delta)^{\varepsilon}}\,\d x\\
 &\quad +\left(\frac{2r\delta_1^{p'}}{\vep p'}+\frac{C}{\alpha\overline{(p-1)}-\varepsilon}+\frac{C}{\alpha\overline{(r-1)}-\varepsilon}\right)\int_{\Omega_{R'}}\frac{\beta_2^{p'}{\rho}^{2r}}{(Mh_\delta)^{\varepsilon}}\,\d x\\
 &\quad +\left(\frac{CC_2}{\alpha\overline{(p-1)}-\varepsilon}+\frac{CC_2}{\alpha\overline{(r-1)}-\varepsilon}\right)\int_{\Omega_{R'}}\frac{|\nabla u|^p{\rho}^{2r}}{(Mh_\delta)^{\varepsilon}}\,\d x\\
 &\quad+\frac{C}{\alpha\overline{(p-1)}-\varepsilon}\int_{\Omega_{R'}}\frac{|\nabla u|^p{\rho}^{p}}{(Mh_\delta)^{\varepsilon}}\,\d x\\
 &\quad +\left(\frac{1}{\varepsilon r'}+\frac{C}{\alpha\overline{(r-1)}-\varepsilon}\right)\int_{\Omega_{R'}}\frac{|g|^r\rho^{2r}}{(Mh_\delta)^\varepsilon}\,\d x\\
 &\quad +\frac{C}{\alpha\overline{(r-1)}-\varepsilon}\int_{\Omega_{R'}}\left(\frac{|u|^{r}\rho^{r}}{(Mh_\delta)^{\varepsilon}}+\frac{|u|^{r}\rho^{2r}}{(Mh_\delta)^{\varepsilon}}\right)\,\d x\\
 &\quad +\frac{C}{\alpha\overline{(r-1)}-\varepsilon}\int_{\Omega_{R'}}\frac{|\nabla{\rho}|^{p'r'/(p'-r')}}{(Mh_\delta)^{\varepsilon}}\,\d x.
\end{align*}
Here we also used the fact that $(2r-1)p' > 2(r-1)r' = 2r$. Consequently, \prf{we obtain
\begin{align*}
 \MoveEqLeft{
 \left[C_1-\left(\frac{2rC_2\delta_1^{p'}}{p'}+\frac{\delta_2^p}{p}\right)-\left(\frac{CC_2\varepsilon}{\alpha\overline{(p-1)}-\varepsilon}+\frac{CC_2\varepsilon}{\alpha\overline{(r-1)}-\varepsilon}\right)\right]\int_{\Omega_{R'}}\frac{|\nabla u|^{p}\rho^{2r}}{(Mh_\delta)^\varepsilon}\,\d x
 }\\
 &\quad +\left(\frac{1}{r'}-\frac{C\varepsilon}{\alpha\overline{(r-1)}-\varepsilon}\right)\int_{\Omega_{R'}}\frac{|u|^{r}\rho^{2r}}{(Mh_\delta)^\varepsilon}\,\d x\\%
 &\leq \frac{C\vep}{\alpha\overline{(p-1)}-\varepsilon}\int_{\Omega_{R'}}\frac{|\nabla u|^p{\rho}^{p}}{(Mh_\delta)^{\varepsilon}}\,\d x+\frac{C\varepsilon}{\alpha\overline{(r-1)}-\varepsilon}\int_{\Omega_{R'}}\frac{|u|^{r}\rho^{r}}{(Mh_\delta)^\varepsilon}\,\d x\\
 &\quad+\int_{\Omega_{R'}}\frac{\beta_1\rho^{2r}}{(Mh_\delta)^\varepsilon}\,\d x+ \left[ \frac{2r}{p}\left(\frac{2}{\delta_1^p}+1\right) + \frac{C\vep}{\alpha\overline{(p-1)}-\varepsilon} \right] \int_{\Omega_{R'}}\frac{|u|^p|\nabla \rho|^p}{(Mh_\delta)^\varepsilon}\,\d x\\
 &\quad +\left(\frac{1}{p'\delta_2^{p'}}+\frac{2r}{p'}+\frac{C\varepsilon}{\alpha\overline{(p-1)}-\varepsilon}+\frac{C\varepsilon}{\alpha\overline{(r-1)}-\varepsilon}\right)\int_{\Omega_{R'}}\frac{|f|^{p}\rho^{2r}}{(Mh_\delta)^{\varepsilon}}\,\d x\\
 &\quad +\left(\frac{2r\delta_1^{p'}}{p'}+\frac{C\varepsilon}{\alpha\overline{(p-1)}-\varepsilon}+\frac{C\varepsilon}{\alpha\overline{(r-1)}-\varepsilon}\right)\int_{\Omega_{R'}}\frac{\beta_2^{p'}{\rho}^{2r}}{(Mh_\delta)^{\varepsilon}}\,\d x\\
 &\quad +\left(\frac{1}{r'}+\frac{C\varepsilon}{\alpha\overline{(r-1)}-\varepsilon}\right)\int_{\Omega_{R'}}\frac{|g|^r\rho^{2r}}{(Mh_\delta)^\varepsilon}\,\d x\\
 &\quad +\frac{C\varepsilon}{\alpha\overline{(r-1)}-\varepsilon}\int_{\Omega_{R'}}\frac{|\nabla{\rho}|^{p'r'/(p'-r')}}{(Mh_\delta)^{\varepsilon}}\,\d x,
\end{align*}
which along with }the fact that ${\rho}|_{\Omega_R}\equiv 1$, $0\leq{\rho}\leq 1$ in $\mathbb{R}^d$ yields
\begin{align*}
 \MoveEqLeft{
 \left[C_1-\left(\frac{2rC_2\delta_1^{p'}}{p'}+\frac{\delta_2^p}{p}\right)-\left(\frac{CC_2\varepsilon}{\alpha\overline{(p-1)}-\varepsilon}+\frac{CC_2\varepsilon}{\alpha\overline{(r-1)}-\varepsilon}\right)\right]\int_{\Omega_{R}}\frac{|\nabla u|^{p}}{(Mh_\delta)^\varepsilon}\,\d x
 }\\
 &\quad +\left(\frac{1}{r'}-\frac{C\varepsilon}{\alpha\overline{(r-1)}-\varepsilon}\right)\int_{\Omega_{R}}\frac{|u|^{r}}{(Mh_\delta)^\varepsilon}\,\d x\\%
 &\leq\frac{C\varepsilon}{\alpha\overline{(p-1)}-\varepsilon}\int_{\Omega_{R'}}\frac{|\nabla u|^p}{(Mh_\delta)^{\varepsilon}}\,\d x+\frac{C\varepsilon}{\alpha\overline{(r-1)}-\varepsilon}\int_{\Omega_{R'}}\frac{|u|^{r}}{(Mh_\delta)^\varepsilon}\,\d x\\
 &\quad +\int_{\Omega_{R'}} \frac{\beta_1}{(Mh_\delta)^\varepsilon}\,\d x + \left[ \frac{2r}{p}\left(\frac{2}{\delta_1^p}+1\right) + \frac{C\vep}{\alpha\overline{(p-1)}-\varepsilon} \right] \int_{\Omega_{R'}}\frac{|u|^p|\nabla \rho|^p}{(Mh_\delta)^\varepsilon}\,\d x\\
 &\quad +\left( \frac{1}{p'\delta_2^{p'}}+\frac{2r}{p'}+\frac{C\varepsilon}{\alpha\overline{(p-1)}-\varepsilon}+\frac{C\varepsilon}{\alpha\overline{(r-1)}-\varepsilon}\right) \int_{\Omega_{R'}}\frac{|f|^{p}}{(Mh_\delta)^{\varepsilon}}\,\d x\\
 &\quad +\left(\frac{2r\delta_1^{p'}}{p'}+\frac{C\varepsilon}{\alpha\overline{(p-1)}-\varepsilon}+\frac{C\varepsilon}{\alpha\overline{(r-1)}-\varepsilon}\right)\int_{\Omega_{R'}}\frac{\beta_2^{p'}}{(Mh_\delta)^{\varepsilon}}\,\d x\\
 &\quad +\left(\frac{1}{r'}+\frac{C\varepsilon}{\alpha\overline{(r-1)}-\varepsilon}\right)\int_{\Omega_{R'}}\frac{|g|^r}{(Mh_\delta)^\varepsilon}\,\d x\\
 &\quad +\frac{C\varepsilon}{\alpha\overline{(r-1)}-\varepsilon}\int_{\Omega_{R'}}\frac{|\nabla{\rho}|^{p'r'/(p'-r')}}{(Mh_\delta)^{\varepsilon}}\,\d x.
\end{align*}

Now, take $\delta_1,\delta_2>0$ small enough so that
$$
C_1-\left(\frac{2rC_2\delta_1^{p'}}{p'}+\frac{\delta_2^p}{p}\right)>0.
$$
Moreover, \prf{we see that
\begin{align*}
C_1-\left(\frac{2rC_2\delta_1^{p'}}{p'}+\frac{\delta_2^p}{p}\right)-\left(\frac{CC_2\varepsilon}{\alpha\overline{(p-1)}-\varepsilon}+\frac{CC_2\varepsilon}{\alpha\overline{(r-1)}-\varepsilon}\right) &\to C_1-\left(\frac{2rC_2\delta_1^{p'}}{p'}+\frac{\delta_2^p}{p}\right),\\
\frac{1}{r'}-\frac{C\varepsilon}{\alpha\overline{(r-1)}-\varepsilon} &\to \frac{1}{r'}
\end{align*}
as $\vep \to 0_+$, and therefore} one can take $\varepsilon_\alpha>0$ such that
\begin{align*}
 C_1-\left(\frac{2rC_2\delta_1^{p'}}{p'}+\frac{\delta_2^p}{p}\right)-\left(\frac{CC_2\varepsilon}{\alpha\overline{(p-1)}-\varepsilon}+\frac{CC_2\varepsilon}{\alpha\overline{(r-1)}-\varepsilon}\right)&>0,\\
 \frac{1}{r'}-\frac{C\varepsilon}{\alpha\overline{(r-1)}-\varepsilon}&>0
\end{align*}
for any $\varepsilon\in(0,\varepsilon_\alpha)$. Hence for each $\varepsilon\in(0,\varepsilon_\alpha)$, we conclude that
\begin{align*}
 \MoveEqLeft{
 \int_{\Omega_R}\frac{|\nabla u|^{p}}{(Mh_\delta)^\varepsilon}\,\d x+\int_{\Omega_{R}}\frac{|u|^{r}}{(Mh_\delta)^\varepsilon}\,\d x
 }\\%
 &\leq C \varepsilon \left( \int_{\Omega_{R'}}\frac{|\nabla u|^p}{(Mh_\delta)^{\varepsilon}}\,\d x + \int_{\Omega_{R'}}\frac{|u|^{r}}{(Mh_\delta)^\varepsilon}\,\d x \right)\\
 &\quad + C (\vep+1) \int_{\Omega_{R'}}\frac{|u|^{p}|\nabla {\rho}|^p}{(Mh_\delta)^{\varepsilon}}\,\d x\\
 &\quad +C\int_{\Omega_{R'}}\frac{\beta_1}{(Mh_\delta)^\varepsilon}\,\d x + C\int_{\Omega_{R'}}\frac{\beta_2^{p'}}{(Mh_\delta)^{\varepsilon}}\,\d x + C\int_{\Omega_{R'}}\frac{|f|^{p}}{(Mh_\delta)^{\varepsilon}}\,\d x\\
 &\quad +C\int_{\Omega_{R'}}\frac{|g|^r}{(Mh_\delta)^\varepsilon}\,\d x+C\varepsilon\int_{\Omega_{R'}}\frac{|\nabla{\rho}|^{p'r'/(p'-r')}}{(Mh_\delta)^{\varepsilon}}\,\d x.
\end{align*}

Now we set $R_n:=R\sum_{i=0}^{n}2^{-i} \in [R, 2R]$ for $n\inz_{\geq 0}$ and let $\zeta_n\in C_c^\infty(\mathbb{R}^d)$ be such that
$$
\zeta_n|_{B_{R_n}}\equiv 1,\quad \mbox{ supp}\,\zeta_n\subset\overline{B_{R_{n+1}}},\quad \sup_{x\inr^d}|\nabla \zeta_n(x)|\leq C\frac{2^{n+1}}{R}
$$
for some constant $C$ independent of $n$ and $R$. Applying the preceding argument to the choice $(R,R',{\rho}):=(R_n,R_{n+1},\zeta_n)$, we find that
\begin{align*}
 \MoveEqLeft{\int_{\Omega_{R_n}}\frac{|\nabla u|^{p}}{(Mh_\delta)^\varepsilon}\,\d x+\int_{\Omega_{R_n}}\frac{|u|^{r}}{(Mh_\delta)^\varepsilon}\,\d x}\\%
 \prf{ &\leq C\varepsilon\left(\int_{\Omega_{R_{n+1}}}\frac{|\nabla u|^p}{(Mh_\delta)^{\varepsilon}}\,\d x+\int_{\Omega_{R_{n+1}}}\frac{|u|^{r}}{(Mh_\delta)^\varepsilon}\,\d x\right)\\
 &\quad + C(\vep+1)\left(\frac{2^{n+1}}{R}\right)^p\int_{\Omega_{R_{n+1}}}\frac{|u|^p}{(Mh_\delta)^\varepsilon}\,\d x\\
 &\quad + C\int_{\Omega_{R_{n+1}}}\frac{\beta_1}{(Mh_\delta)^\varepsilon}\,\d x + C\int_{\Omega_{R_{n+1}}}\frac{\beta_2^{p'}}{(Mh_\delta)^{\varepsilon}}\,\d x\\
 &\quad + C \int_{\Omega_{R_{n+1}}}\frac{|f|^{p}}{(Mh_\delta)^{\varepsilon}}\,\d x + C\int_{\Omega_{R_{n+1}}}\frac{|g|^r}{(Mh_\delta)^\varepsilon}\,\d x\\
 &\quad + C\varepsilon \left(\frac{2^{n+1}}R\right)^{p'r'/(p'-r')} \int_{\Omega_{R_{n+1}}}\frac{1}{(Mh_\delta)^{\varepsilon}}\,\d x\\
 &\leq C\varepsilon\left(\int_{\Omega_{R_{n+1}}}\frac{|\nabla u|^p}{(Mh_\delta)^{\varepsilon}}\,\d x+\int_{\Omega_{R_{n+1}}}\frac{|u|^{r}}{(Mh_\delta)^\varepsilon}\,\d x\right)\\
 &\quad +C(\vep+1)\left(\frac{2^{n+1}}{R}\right)^p\int_{\Omega_{R_{n+1}}}\frac{|u|^p}{(Mh_\delta)^{\varepsilon p/r}}\frac{1}{(Mh_\delta)^{\varepsilon(1-p/r)}}\,\d x\\
 &\quad +C\int_{\Omega_{R_{n+1}}}\frac{\beta_1}{(Mh_\delta)^\varepsilon}\,\d x + C\int_{\Omega_{R_{n+1}}}\frac{\beta_2^{p'}}{(Mh_\delta)^{\varepsilon}}\,\d x\\
 &\quad +C\int_{\Omega_{R_{n+1}}}\frac{|f|^{p}}{(Mh_\delta)^{\varepsilon}}\,\d x + C\int_{\Omega_{R_{n+1}}}\frac{|g|^r}{(Mh_\delta)^\varepsilon}\,\d x\\
 &\quad +C\varepsilon \left(\frac{2^{n+1}}R\right)^{p'r'/(p'-r')} \int_{\Omega_{R_{n+1}}} \frac{1}{(Mh_\delta)^{\varepsilon}}\,\d x\\}
 &\leq C\varepsilon\left(\int_{\Omega_{R_{n+1}}}\frac{|\nabla u|^p}{(Mh_\delta)^{\varepsilon}}\,\d x+\int_{\Omega_{R_{n+1}}}\frac{|u|^{r}}{(Mh_\delta)^\varepsilon}\,\d x\right)\\
 &\quad +C\left(\frac{2^{n+1}}{R}\right)^{p}\left(\int_{\Omega_{R_{n+1}}}\frac{|u|^r}{(Mh_\delta)^{\varepsilon}}\,\d x\right)^{p/r}\left(\int_{\Omega_{R_{n+1}}}\frac{1}{(Mh_\delta)^{\varepsilon}}\,\d x\right)^{(r-p)/r}\\
 &\quad +C\int_{\Omega_{R_{n+1}}}\frac{\beta_1}{(Mh_\delta)^\varepsilon}\,\d x + C\int_{\Omega_{R_{n+1}}}\frac{\beta_2^{p'}}{(Mh_\delta)^{\varepsilon}}\,\d x\\
 &\quad +C\int_{\Omega_{R_{n+1}}}\frac{|f|^{p}}{(Mh_\delta)^{\varepsilon}}\,\d x + C\int_{\Omega_{R_{n+1}}}\frac{|g|^r}{(Mh_\delta)^\varepsilon}\,\d x\\
 &\quad+C\varepsilon \left(\frac{2^{n+1}}R\right)^{p'r'/(p'-r')} \int_{\Omega_{R_{n+1}}}\frac{1}{(Mh_\delta)^{\varepsilon}}\,\d x\\
 &\leq C\varepsilon\left(\int_{\Omega_{R_{n+1}}}\frac{|\nabla u|^p}{(Mh_\delta)^{\varepsilon}}\,\d x + \int_{\Omega_{R_{n+1}}}\frac{|u|^{r}}{(Mh_\delta)^\varepsilon}\,\d x \right) + \nu \int_{\Omega_{R_{n+1}}}\frac{|u|^{r}}{(Mh_\delta)^\varepsilon}\,\d x \\
 &\quad + C(\vep + \nu^{-p/(r-p)}) \delta^{-\vep} \left(\frac{2^{n+1}}{R}\right)^{pr/(r-p)}|\Omega_{2R}| \\
 &\quad + C\int_{\Omega_{R_{n+1}}}\frac{\beta_1}{(Mh_\delta)^\varepsilon}\,\d x + C \int_{\Omega_{R_{n+1}}}\frac{\beta_2^{p'}}{(Mh_\delta)^{\varepsilon}}\,\d x\\
 &\quad + C\int_{\Omega_{R_{n+1}}}\frac{|f|^{p}}{(Mh_\delta)^{\varepsilon}}\,\d x + C\int_{\Omega_{R_{n+1}}}\frac{|g|^r}{(Mh_\delta)^\varepsilon}\,\d x 
\end{align*}
for $\nu > 0$, which will be chosen later. Here the last inequality follows from the fact that $Mh_\delta\geq 1$ and $p'r'/(p'-r')=pr/(r-p)$ as well as Young's inequality. For each $n\inz_{\geq 0}$, setting
$$
H_n :=\int_{\Omega_{R_n}}\frac{|\nabla u|^{p}}{(Mh_\delta)^\varepsilon}\,\d x+\int_{\Omega_{R_n}}\frac{|u|^{r}}{(Mh_\delta)^\varepsilon}\,\d x,
$$
we obtain
\begin{align*}
 H_n \prf{&\leq \left(C\varepsilon+\nu\right) H_{n+1}+C(\vep + \nu^{-p/(r-p)}) \delta^{-\vep} \left(\frac{2^{n+1}}{R}\right)^{pr/(r-p)}|\Omega_{2R}|\\
&\quad +C\int_{\Omega_{R_{n+1}}}\frac{\beta_1}{(Mh_\delta)^\varepsilon}\,\d x +C\int_{\Omega_{R_{n+1}}}\frac{\beta_2^{p'}}{(Mh_\delta)^{\varepsilon}}\,\d x\\
&\quad + C\int_{\Omega_{R_{n+1}}}\frac{|f|^{p}}{(Mh_\delta)^{\varepsilon}}\,\d x + C\int_{\Omega_{R_{n+1}}}\frac{|g|^r}{(Mh_\delta)^\varepsilon}\,\d x\\}
 &\leq \left(C\varepsilon+\nu\right) H_{n+1} + C(\vep + \nu^{-p/(r-p)}) \delta^{-\vep} \left(\frac{2^{n+1}}{R}\right)^{pr/(r-p)}|\Omega_{2R}| \\
 &\quad + C\int_{\Omega_{2R}}\frac{\beta_1}{(Mh_\delta)^\varepsilon}\,\d x + C\int_{\Omega_{2R}}\frac{\beta_2^{p'}}{(Mh_\delta)^{\varepsilon}}\,\d x \\
 &\quad + C\int_{\Omega_{2R}}\frac{|f|^{p}}{(Mh_\delta)^{\varepsilon}}\,\d x + C\int_{\Omega_{2R}}\frac{|g|^r}{(Mh_\delta)^\varepsilon}\,\d x\\
 &=: \left(C\varepsilon+\nu\right) H_{n+1}+C_{R} 2^{npr/(r-p)}+F_0,
\end{align*}
where $C_{R} = C (\vep + \nu^{-p/(r-p)}) \delta^{-\vep} R^{d - \frac{pr}{r-p}}>0$ is a constant independent of $n$ and $F_0 > 0$ is a constant independent of $n,R$ and given as
\begin{align*}
F_0 &:= C \left( \int_{\Omega_{2R}} \beta_1\,\d x + \int_{\Omega_{2R}} \beta_2^{p'}\,\d x + \int_{\Omega_{2R}}\frac{|f|^{p}}{(Mh_\delta)^{\varepsilon}}\,\d x + \int_{\Omega_{2R}}\frac{|g|^r}{(Mh_\delta)^\varepsilon}\,\d x \right).
\end{align*}
Therefore setting $\gamma:=2^{pr/(r-p)} > 1$, we get
$$
H_n\leq \left(C\varepsilon+\nu\right) H_{n+1}+C_{R}\gamma^n+F_0
$$
for any $n \inz_{\geq 0}$, and moreover, we deduce that
\begin{align*}
 H_0 \prf{&\leq \left(C\varepsilon+\nu\right) H_{1}+C_{R}+F_0\\
 &\leq \left(C\varepsilon+\nu\right)^nH_n+C_{R} \sum_{i=0}^{n-1}\left(C\varepsilon+\nu\right)^i\gamma^i+F_0\sum_{i=0}^{n-1}\left(C\varepsilon+\nu\right)^i\\}
 &\leq \left(C\varepsilon+\nu\right)^n\left(\int_{\Omega_{2R}}\frac{|\nabla u|^{p}}{(Mh_\delta)^{\varepsilon}}\,\d x+\int_{\Omega_{2R}}\frac{|u|^{r}}{(Mh_\delta)^{\varepsilon}}\,\d x\right)\\
 &\quad + C_{R}\frac{1-(C\varepsilon+\nu)^n\gamma^n}{1-(C\varepsilon+\nu)\gamma}+F_0\frac{1-(C\varepsilon+\nu)^n}{1-(C\varepsilon+\nu)}\\
 &\to \frac{C_{R}}{1-(C\varepsilon+\nu)\gamma}+\frac{F_0}{1-(C\varepsilon+\nu)}\quad\mbox{as }n\to\infty
\end{align*}
by setting $\nu = 1/(2\gamma)$ and $\vep \in (0,\vep_1)$ with the choice,
$$
\varepsilon_1:=\min\left\{ (2C\gamma)^{-1},\varepsilon_\alpha\right\} > 0
$$
(then $C\vep + \nu < (C\vep+\nu)\gamma < 1$). Thus the desired conclusion follows. This completes the proof.
\end{proof}

\section{Proofs of Theorem \ref{thm:elliplocDS} and Corollary \ref{C:main}}\label{S:proof}

We are now in a position to prove Theorem \ref{thm:elliplocDS}.

\begin{proof}[Proof of Theorem {\rm \ref{thm:elliplocDS}}]
Let $\varepsilon_0>0$ be such that
$$
\varepsilon_0:=\min\left\{\varepsilon_1,\frac{p}{r}\right\} \in (0,1),
$$
where $\varepsilon_1$ is the constant given in Proposition \ref{prop:LocEstimate}. Let $\varepsilon\in(0,\varepsilon_0)$ be fixed arbitrarily. Set $q := p - \varepsilon \in (p-1,p)$ and $s := r(p-\vep)/p \in (r-1,r)$. Moreover, let $f\in L_{{\loc}}^q(\overline{\Omega};\mathbb{R}^{d\times N})$, $g\in L_{{\loc}}^s(\overline{\Omega};\mathbb{R}^{N})$ and fix $R>0$. 

For each $k \inz_{\geq 1}$ (large enough if necessary), we define
\begin{align*}
f^k &:= \min\{ k,|f|\} \frac{f}{|f|} \, \chi_{B_k}\in L^p(\Omega;\mathbb{R}^{d\times N}),\\
g^k &:= \min\{ k,|g|\} \frac{g}{|g|} \, \chi_{B_k}\in L^r(\Omega;\mathbb{R}^{N})
\end{align*}
and denote by $u^k\in X^{p,r}(\Omega;\mathbb{R}^N)$ the unique weak solution to the approximate problem ${\sf (P)}_k$,
$$
\left\{
\begin{alignedat}{4}
&-\,\dv \,A(x,\nabla u^k)+|u^k|^{r-2}u^k = -\,\dv \,(|f^k|^{p-2}f^k)+|g^k|^{r-2}g^k \ &&\mbox{ in } \
\Omega,\\&
u^k= 0 \ &&\mbox{ on } \ \partial \Omega.
\end{alignedat}
\right.
$$
We apply Proposition \ref{prop:LocEstimate} to \Pk \ along with 
$$
\delta := 1 \quad \mbox{ and } \quad h := |f|+|g|^{s/q} \in L^q_{\rm loc}(\overline{\Omega}).
$$
It then follows that
\begin{align}
\MoveEqLeft{\int_{\Omega_{R}}\frac{|\nabla u^k|^{p}}{(M[{h}\chi_{\Omega_{2R}}]+1)^\varepsilon}\,\d x+\int_{\Omega_{R}}\frac{|u^k|^{r}}{(M[{h}\chi_{\Omega_{2R}}]+1)^\varepsilon}\,\d x}\notag\\
&\leq C_0 \left( \int_{\Omega_{2R}}\frac{|f^k|^{p}}{(M[{h}\chi_{\Omega_{2R}}]+1)^\varepsilon}\,\d x + \int_{\Omega_{2R}}\frac{|g^k|^{r}}{(M[{h}\chi_{\Omega_{2R}}]+1)^\varepsilon}\,\d x \right.\notag\\
&\quad + \left. \int_{\Omega_{2R}}\frac{\beta_1}{(M[{h}\chi_{\Omega_{2R}}]+1)^\varepsilon}\,\d x + \int_{\Omega_{2R}}\frac{\beta_2^{p'}}{(M[{h}\chi_{\Omega_{2R}}]+1)^\varepsilon}\,\d x \right)\notag\\
&\quad + C_R. \label{loc_en_est_k}
\end{align}
Here and henceforth, we set a weight function as
\begin{equation}\label{omega}
\omega := (M[{h}\chi_{\Omega_{2R}}]+1)^{-\varepsilon} = M[{h}\chi_{\Omega_{2R}}+1]^{-\varepsilon}. 
\end{equation}

Then we have
\begin{lemma}\label{lem:ellipuksomebddness}
For each $R>0$, $(\nabla u^k)_k$ is bounded both in $L^p_\omega(\Omega_R;\mathbb{R}^{d\times N})$ and in $L^q(\Omega_R;\mathbb{R}^{d\times N})$. Moreover, $(u^k)_k$ is bounded both in $L^r_\omega(\Omega_R;\mathbb{R}^N)$ and in $L^s(\Omega_R;\mathbb{R}^N)$. 
\end{lemma}

\begin{proof}
For each $R > 0$, we observe that
$$
 M[h\chi_{\Omega_{2R}}] + 1 \geq M[|f|\chi_{\Omega_{2R}}] \geq |f|\chi_{\Omega_{2R}} = |f| \ \mbox{ a.e.~in } \Omega_{2R},
$$
which along with \eqref{omega} and $\vep = p-q$ implies
\begin{equation}\label{cut-f}
\int_{\Omega_{2R}} |f^k|^p \omega \, \d x \leq \int_{\Omega_{2R}} |f|^p \omega \, \d x \leq \int_{\Omega_{2R}} |f|^{q} \, \d x < +\infty
\end{equation}
for $k \inz_{\geq 1}$. Similarly, noting that $M[h\chi_{\Omega_{2R}}] + 1 \geq |g|^{s/q}$ a.e.~in $\Omega_{2R}$ and $r - \vep s/q = s$, we find that 
\begin{equation}\label{cut-g}
\int_{\Omega_{2R}} |g^k|^r \omega \, \d x \leq \int_{\Omega_{2R}} |g|^r \omega \, \d x \leq \int_{\Omega_{2R}} |g|^{s} \, \d x < +\infty
\end{equation}
for $k \inz_{\geq 1}$. Hence we observe from \eqref{loc_en_est_k} that $(\nabla u^k)_k$ is bounded in $L^p_\omega(\Omega_R;\mathbb{R}^{d\times N})$ and $(u^k)_k$ is bounded in $L^r_\omega(\Omega_R;\mathbb{R}^N)$ as $k \to \infty$. 

Thanks to Lemma \ref{lem:Emb_Mf}, since $\omega^{-1} = (M[h \chi_{\Omega_{2R}}] + 1)^{p-q} \in L^{q/(p-q)}_{\rm loc}(\R^d) = L^{s/(r-s)}_{\rm loc}(\R^d)$ (indeed, $h \in L^q_{\rm loc}(\overline{\Omega})$ and $M$ is bounded in $L^q(\R^d)$) and $q/(p-q) = s/(r-s)$ from $s = rq/p$, both embeddings $L^p_\omega(\Omega_R;\mathbb{R}^{d\times N}) \hookrightarrow L^q(\Omega_R;\mathbb{R}^{d\times N})$ and $L^r_\omega(\Omega_R;\mathbb{R}^{N}) \hookrightarrow L^s(\Omega_R;\mathbb{R}^{N})$ are continuous. Therefore $(\nabla u^k)_k$ is bounded in $L^q(\Omega_R;\mathbb{R}^{d\times N})$, and $(u^k)_k$ is so in $L^s(\Omega_R;\mathbb{R}^N)$. 
\end{proof}

From the reflexivity, one can take a (not relabeled) subsequence and measurable functions $u: \R^d \to \R^N$, $v_1 : \R^d \to \R^{d\times N}$ and $v_2 : \R^d \to \R^N$ such that $u\in W^{1,p}_{\omega}(\Omega_R;\mathbb{R}^N)\cap L^r_{\omega}(\Omega_R;\mathbb{R}^N)$, $v_1\in L^{p'}_{\omega}(\Omega_R;\mathbb{R}^{d\times N})$, $v_2\in L^{r'}_{\omega}(\Omega_R;\mathbb{R}^N)$ and
\begin{alignat}{4}
u^k&\to u\quad  &&{\rm weakly\ in}\ W^{1,p}_\omega(\Omega_{R};\mathbb{R}^N),\label{c:W1pom}\\
& &&{\rm weakly\ in}\ L^r_\omega(\Omega_{R};\mathbb{R}^N),\notag\\
A(x,\nabla u^k)&\to v_1\quad && {\rm weakly\ in}\ L^{p'}_\omega(\Omega_{R};\mathbb{R}^{d\times N}),\label{conv:A(x,nabla u^k)weakly)}\\
|u^k|^{r-2}u^k&\to v_2\quad && {\rm weakly\ in}\ L^{r'}_\omega(\Omega_{R};\mathbb{R}^N),\label{conv:|u^k|^r-1weakly}
\end{alignat}
for any $R > 0$. Here we remark that $u$, $v_1$ and $v_2$ can be constructed being independent of $R > 0$ by a diagonal argument. Moreover, one has

\begin{lemma}\label{lem:Range}
For any $\varphi\in C_c^\infty(\Omega;\mathbb{R}^N)$, it holds that
\begin{align}
\lim_{k \to\infty} \int_{\Omega}|u^k|^{r-2}u^k\cdot\varphi\,\d x
&= \int_{\Omega}v_2\cdot\varphi\,\d x,\label{c1}\\
\lim_{k \to\infty}\int_{\Omega}A(x,\nabla u^k):\nabla\varphi\,\d x&=\int_{\Omega}v_1:\nabla\varphi\,\d x.\label{c2}
\end{align}
\end{lemma}

\begin{proof}
Let $\varphi\in C_c^\infty(\Omega;\mathbb{R}^N)$ and $R > 0$ be such that $\mbox{supp}\,\varphi\subset\Omega_R$. Noting that $\omega^{1-r} \in L^{s/((r-s)(r-1))}(\Omega_R)$ with $s/((r-s)(r-1)) > 1$ by $s > r-1$, we find that
$$
\int_{\Omega_R}|\varphi\omega^{-1}|^r\omega\,\d x = \int_{\Omega_R}|\varphi|^r\omega^{1-r}\,\d x<\infty,
$$
which yields $\varphi\omega^{-1}\in L^r_\omega(\Omega_R;\mathbb{R}^N)$. Moreover, it follows from \eqref{conv:|u^k|^r-1weakly} that
\begin{align*}
 \int_{\Omega}|u^k|^{r-2}u^k\cdot\varphi\,\d x
 &=\int_{\Omega_R}|u^k|^{r-2}u^k\cdot\varphi\,\d x\\
 &=\int_{\Omega_R}|u^k|^{r-2}u^k\cdot(\varphi\omega^{-1})\omega\,\d x\\
 &\to\int_{\Omega_R}v_2\cdot(\varphi\omega^{-1})\omega\,\d x
 =\int_{\Omega}v_2\cdot\varphi\,\d x,
\end{align*}
which implies \eqref{c1}. Similarly, \eqref{c2} also follows for any $\varphi\in C_c^\infty(\Omega;\mathbb{R}^N)$. Indeed, we recall that $\omega^{1-p} \in L^{q/((p-q)(p-1))}(\Omega_R)$, which implies $\varphi \omega^{-1} \in L^p_\omega(\Omega_R;\R^N)$. Thus the proof is complete.
\end{proof}

Now, we recall that $u^k$ is the unique weak solution to ${\sf (P)}_k$, that is,
\begin{align}
\lefteqn{
\int_\Omega A(x,\nabla u^k) : \nabla\varphi\,\d x + \int_\Omega|u^k|^{r-2}u^k \cdot \varphi\,\d x
}\notag\\
&= \int_\Omega|f^k|^{p-2}f^k : \nabla\varphi\,\d x + \int_\Omega|g^k|^{r-2}g^k \cdot \varphi\,\d x\label{k}
\end{align}
for any $\varphi\in C_c^\infty(\Omega;\mathbb{R}^N)$. Passing to the limit as $k\to\infty$, we deduce that
\begin{equation}\label{WeakLim}
\int_\Omega v_1:\nabla\varphi\,\d x+\int_\Omega v_2\cdot\varphi\,\d x= \int_\Omega|f|^{p-2}f:\nabla\varphi\,\d x+\int_\Omega|g|^{r-2}g\cdot\varphi\,\d x
\end{equation}
for any $\varphi\in C_c^\infty(\Omega;\mathbb{R}^N)$. 

We next identify the weak limits. First of all, we claim that
\begin{equation}\label{claim1}
v_2=|u|^{r-2}u \quad \mbox{ a.e.~in } \Omega.
\end{equation}
Indeed, let $n \in \mathbb{N}$ and let $\Omega^{(n)}:=\{x\in\Omega\,:\,\mbox{dist}(x,\partial\Omega)>1/n\}$. Then the set $\overline{B_R} \cap \overline{\Omega^{(n)}}$ is compact in $\mathbb{R}^d$ for each $R>0$. Since $(B_{1/(2n)}(x))_{x\in\overline{B_R}\cap\overline{\Omega^{(n)}}}$ is an open covering of the set $\overline{B_R} \cap \overline{\Omega^{(n)}}$, one can take $m(n) \inz_{\geq 1}$ and $x_1^{(n)},\ldots,x_{m(n)}^{(n)}\in\overline{B_R}\cap\overline{\Omega^{(n)}}$ such that $\overline{B_R}\cap\overline{\Omega^{(n)}} \subset \bigcup_{i=1}^{m(n)}B_{1/(2n)}(x_i^{(n)})$. We note that $B_{1/(2n)}(x_i^{(n)})\subset\Omega_{2R}$ for each $i$. Therefore due to Proposition \ref{prop:LocEstimate}, $(u^k)$ turns out to be bounded in $W^{1,q}(B_{1/(2n)}(x_i^{(n)}))$ (uniformly in $k$). Moreover, $W^{1,q}(B_{1/(2n)}(x_i^{(n)}))$ is compactly embedded in $L^{q}(B_{1/(2n)}(x_i^{(n)}))$, we can take a (not relabeled) subsequence of $(u^k)$ such that $u^k \to u$ a.e.~in $B_{1/(2n)}(x_i^{(n)})$. Repeating this procedure $m(n)$-times, we eventually take a subsequence so that $u^k \to u$ a.e.~in $B_R \cap \Omega^{(n)}$. Hence by a diagonal argument, up to a subsequence, $u^k \to u$ a.e.~in $\bigcup_{n\inz_{\geq 1}}\left( B_R \cap \Omega^{(n)} \right)$. Since $\Omega_R\subset\bigcup_{n\inz_{\geq 1}}\left( B_R \cap \Omega^{(n)} \right)$, we particularly observe that
\begin{equation}\label{c:pnt}
u^k\to u\quad\mbox{a.e. in }\Omega_R.
\end{equation}
Therefore, by continuity, for each $R>0$, we have $|u^k|^{r-2}u^k\to |u|^{r-2}u$ a.e.~in $\Omega_R$, and hence, we can conclude that $v_2=|u|^{r-2}u$ a.e.~in $\Omega_R$. From the arbitrariness of $R>0$, we obtain \eqref{claim1}.

It still remains to show that
\begin{equation}\label{claim2}
v_1=A(\cdot,\nabla u)\quad \mbox{a.e. in }\Omega,
\end{equation}
which will be proved below. Let $R>0$ be fixed. We shall apply Lemma \ref{lem:WDC} with the choice $D=\Omega_R\subset\mathbb{R}^d$, $a^k=\nabla u^k_l \in L^p_\omega(\D;\R^d)$ and $b^k=[A(x,\nabla u^k)]_{\,\cdot\,l} \in L^{p'}_\omega(\D;\R^d)$ for each component, $l = 1,\ldots,N$. To this end, we note that $\omega \in \mathcal{A}_p$, which can be assured by virtue of Lemma \ref{lem:Muckenhoupt} along with the fact that $\vep = p-q < p-1$ (see \eqref{omega}). Moreover, let us first check the assumptions of Lemma \ref{lem:WDC}.

We can check (i) immediately from Lemma \ref{lem:ellipuksomebddness} together with \eqref{ineq:ellipAbdd}. We next check (ii). Let $\varphi \in W^{1,\infty}_0(\Omega_R)$. Then for each $\sigma \in [1,\infty)$ there exists a sequence $(\varphi_m)$ in $C_c^\infty(\Omega_R)$ such that $\varphi_m \to \varphi$ strongly in $W^{1,\sigma}_0(\Omega_R)$. Here we take $\sigma \in (p',\infty)$ so that $W^{1,p}_\omega(\Omega_R) \hookrightarrow W^{1,\sigma'}(\Omega_R)$ (see Lemma \ref{lem:Emb}). An integration by parts yields
\begin{align*}
\lefteqn{
\int_{\Omega_R}((\nabla u^k_l)_j\partial_i\varphi_m-(\nabla u^k_l)_i\partial_j\varphi_m)\,\d x
}\\
 &= -\int_{\Omega_R}(u^k_l\partial_{j}\partial_{i}\varphi_m-u^k_l\partial_{i}\partial_{j}\varphi_m)\,\d x =0,
\end{align*}
which yields
\begin{align*}
\lefteqn{
\int_{\Omega_R}((\nabla u^k_l)_j\partial_i\varphi-(\nabla u^k_l)_i\partial_j\varphi)\,\d x
}\\
&= \int_{\Omega_R} \left[ (\nabla u^k_l)_j\partial_i(\varphi-\varphi_m)-(\nabla u^k_l)_i\partial_j(\varphi-\varphi_m) \right]\,\d x\\
&\quad + \int_{\Omega_R}((\nabla u^k_l)_j\partial_i\varphi_m-(\nabla u^k_l)_i\partial_j\varphi_m)\,\d x \to 0
\end{align*}
as $m \to \infty$. Thus (ii) follows. Finally, let us check (iii). Let $(c^k)$ be a sequence in $W^{1,\infty}_0(\Omega_R)$ such that $\nabla c^k \to 0$ weakly star in $L^\infty(\Omega_R;\mathbb{R}^d)$. We use the same letter $c^k$ for the zero extension of $c^k$ onto $\Omega \subset \R^d$. Then we see that $c^k\in W^{1,p}_0(\Omega)\cap L^r(\Omega)$ since $c^k$ lies on $W^{1,p}_0(\Omega_R)\cap L^r(\Omega_R)$. Now, for each $k \inz_{\geq 1}$, one can take a test function $\varphi\in X^{p,r}(\Omega;\mathbb{R}^N)$ whose $j$-th  component $\varphi_j$ is given by
\begin{align*}
 \varphi_j(x) &= \begin{cases}
		   c^k(x)& \mbox{if } \ j=l,\\
		   0 & \mbox{if } \ j\neq l
		  \end{cases}
\end{align*}
for $x \in \Omega$. Then we have
\begin{align*}
\MoveEqLeft{
\int_{\Omega_R} A(\cdot,\nabla u^k)_l\cdot\nabla c^k\,\d x+\int_{\Omega_R} (|u^k|^{r-2}u^k)_lc^k\,\d x
}\\
&=\int_{\Omega_R} (|f^k|^{p-2}f^k)_l\cdot\nabla c^k\,\d x+\int_{\Omega_R} (|g^k|^{r-2}g^k)_l c^k\,\d x.
\end{align*}
Since $(c^k)$ is bounded in $W^{1,d+1}_0(\Omega_R)$ from Poincar\'e's inequality, there exists $c \in W^{1,d+1}_0(\Omega_R)$ such that, up to a subsequence, $c^k \to c$ weakly in $W^{1,d+1}_0(\Omega_R)$. Moreover, from the compact embedding $W^{1,d+1}_0(\Omega_R) \hookrightarrow L^{\infty}(\Omega_R)$, taking a (not relabeled) subsequence, we deduce that $c^k \to c$ in $L^{\infty}(\Omega_R)$ and a.e.~in $\Omega_R$. Furthermore, one finds that
$$
\nabla c^k \to \nabla c \quad \mbox{ weakly in } L^{d+1}(\Omega_R;\mathbb{R}^d).
$$
Recalling that $\nabla c^k \to 0$ weakly star in $L^\infty(\Omega_R;\mathbb{R}^d)$, we infer that $\nabla c = 0$ a.e.~in $\Omega_R$. Hence it is constant in $\Omega_R$. As $c \in W^{1,1}_0(\Omega_R)$, Poincar\'e's inequality yields $c=0$. Since $L^{r}_\omega(\Omega_R;\mathbb{R}^N)$ is continuously embedded in $L^{s}(\Omega_R;\mathbb{R}^N)$, it follows from Lemma \ref{lem:ellipuksomebddness} that $(|u^k|^{r-2}u^k)_k$ is bounded in $L^{s/(r-1)}(\Omega_R;\mathbb{R}^N)$. Since $c^k \to 0$ in $L^\infty(\Omega_R)$, we get
$$
\int_{\Omega_R} (|u^k|^{r-2}u^k)_lc^k\,\d x \to 0 \quad \mbox{ as } \ k\to\infty.
$$
Therefore one obtains
\begin{align*}
\MoveEqLeft{
\int_{\Omega_R} A(\cdot,\nabla u^k)_l\cdot\nabla c^k\,\d x
}\\
 &=-\int_{\Omega_R} (|u^k|^{r-2}u^k)_lc^k\,\d x + \int_{\Omega_R} (|f^k|^{p-2}f^k)_l\cdot\nabla c^k\,\d x\\
 &\quad + \int_{\Omega_R} (|g^k|^{r-2}g^k)_lc^k\,\d x
 \to 0\quad\mbox{ as } \ k \to \infty.
\end{align*}
Here we also used the fact that $|f^k|^{p-2}f^k \to |f|^{p-2}f$ strongly in $L^1(\Omega;\R^{d\times N})$ as $k \to \infty$ (indeed, it can be proved from the definition of $f^k$ along with $q > p-1$ and Lebesgue's dominated convergence theorem).

Thus we have checked all the assumptions of Lemma \ref{lem:WDC}, and therefore, we can take a subsequence of $(k)$ and a strictly increasing sequence $(E^j_R)$ of measurable subsets of $\Omega_R$ such that $|\Omega_R \setminus E^j_R| \to 0$ as $j \to \infty$ and
$$
A(\cdot,\nabla u^k)_l\cdot(\nabla u^k)_l\,\omega \to (v_1)_l \cdot(\nabla u)_l \,\omega \quad \mbox{ weakly in } \ L^1(E^j_R)
$$
for all $j\inz_{\geq 1}$ and $l \in \{1,\ldots,N\}$. Therefore summing up both sides in $l$, we deduce that
\begin{equation}\label{conv:ellipticweakconvv1nablau}
 A(\cdot,\nabla u^k):(\nabla u^k)\,\omega \to v_1:(\nabla u)\,\omega \quad \mbox{ weakly in } \ L^1(E^j_R)
\end{equation}
for all $j \inz_{\geq 1}$.

On the other hand, from the monotonicity \eqref{ineq:ellipAmono} of $A$, one derives that
$$
\int_{E^j_R} (A(\cdot,\nabla u^k)-A(\cdot,G)) : (\nabla u^k-G)\omega\,\d x \geq 0
$$
for any $G \in L^p_\omega(\Omega;\mathbb{R}^{d\times N})$. Passing to the limit as $k \to \infty$ and using \eqref{c:W1pom}, \eqref{conv:A(x,nabla u^k)weakly)} and \eqref{conv:ellipticweakconvv1nablau}, we obtain
$$
\int_{E^j_R}(v_1-A(\cdot,G)):(\nabla u-G)\omega\,\d x\geq 0.
$$
Moreover, let $j\to\infty$. It follows that
\begin{equation}\label{1}
\int_{\Omega_R}(v_1-A(\cdot,G)):(\nabla u-G)\omega\,\d x\geq 0.
\end{equation}
Let $\delta>0$ and $H\in L^\infty(\Omega_R;\mathbb{R}^{d\times N})$ be arbitrarily fixed. Set
$$
G=\nabla u-\delta H.
$$
Then $G\in L^p_\omega(\Omega_R;\mathbb{R}^{d\times N})$. We substitute it to $G$ of \eqref{1} to see that
$$
\int_{\Omega_R}(v_1-A(\cdot,\nabla u-\delta H)):\delta H\omega\,\d x\geq 0.
$$
Dividing both sides by $\delta>0$ and passing to the limit as $\delta\to 0$, with the aid of Lebesgue's dominated convergence theorem, we derive that
$$
\int_{\Omega_R}(v_1-A(\cdot,\nabla u)) : H\omega\,\d x\geq 0.
$$
Replacing $H$ with $-H$, we obtain
$$
\int_{\Omega_R}(v_1-A(\cdot,\nabla u)) : H\omega\,\d x=0,
$$
whence it follows that $v_1=A(\cdot,\nabla u)$ a.e.~in $\Omega_R$. From the arbitrariness of $R>0$, we conclude that $v_1=A(\cdot,\nabla u)$ a.e.~in $\Omega$, that is, \eqref{claim2} holds true. 

Thus \eqref{WeakLim} and \eqref{claim1} yield
\begin{align*}
\MoveEqLeft{
\int_\Omega A(\cdot,\nabla u):\nabla\varphi\,\d x+\int_\Omega|u|^{r-2}u\cdot\varphi\,\d x
}\\
&= \int_\Omega|f|^{p-2}f:\nabla\varphi\,\d x+\int_\Omega|g|^{r-2}g\cdot\varphi\,\d x\quad \mbox{ for } \ \varphi\in C_c^\infty(\Omega;\mathbb{R}^N).
\end{align*}
Moreover, thanks to Proposition \ref{prop:LocEstimate}, the weighted estimate \eqref{thm:Localest} follows from the weak lower semicontinuity of the norms.

Finally, we check the boundary condition \BC. Let $R>0$ be fixed and let $\rho_R$ be a cut-off function satisfying \eqref{cut-off}. Since $u^k \in X^{p,r}(\Omega;\mathbb{R}^N)$, we find that $u^k\rho_R \in W^{1,p}_0(\Omega;\mathbb{R}^N)$, which together with the boundedness of $B_R$ also yields $u^k\rho_R \in W^{1,q}_0(\Omega;\mathbb{R}^N)$. Furthermore, as $(u^k\rho_R)_k$ is bounded in $W^{1,q}_0(\Omega;\mathbb{R}^N)$ (see Lemma \ref{lem:ellipuksomebddness}), there exists $v \in W^{1,q}_0(\Omega;\mathbb{R}^N)$ such that
\begin{equation}\label{conv:u^krho_R weakly}
 u^k\rho_R\to v\quad \mbox{ weakly in } \ W^{1,q}_0(\Omega;\mathbb{R}^N)
\end{equation}
up to a subsequence. On the other hand, since we have also seen that $u^k \to u$ a.e.~in $\Omega$ up to a subsequence (see \eqref{c:pnt} with a diagonal argument), we conclude that $u\rho_R=v$ a.e.~in $\Omega$, and hence, $u\rho_R \in W^{1,q}_0(\Omega;\mathbb{R}^N)$, which also implies $u\rho_R \in W^{1,1}_0(\Omega;\mathbb{R}^N)$. Thus \BC \ follows. This completes the proof.
\end{proof}

Then Corollary \ref{C:main} also follows.

\begin{proof}[Proof of Corollary {\rm \ref{C:main}}]
Recalling \eqref{cut-f} and \eqref{cut-g} and utilizing Lemma \ref{lem:Emb_Mf}, one can also derive \eqref{thm:Localest2} from \eqref{thm:Localest}. Indeed, as in \eqref{Lq-Lpom}, one has
\begin{align*}
\int_{\Omega_R} |\nabla u|^q \, \d x 
&\leq \int_{\Omega_R} |\nabla u|^p \omega \, \d x + C \|\omega^{-1}\|_{L^{q/(p-q)}}^{q/(p-q)}\\
&\leq \int_{\Omega_R} |\nabla u|^p \omega \, \d x + C \left( \int_{\Omega_{2R}} |f|^q \, \d x + \int_{\Omega_{2R}} |g|^s \, \d x + \delta^\vep R^d\right)
\end{align*}
for some constant $C$. Optimizing \eqref{thm:Localest} with the above in $\delta$, we can reach the desired conclusion.
\end{proof}

\begin{remark}
The arguments exhibited so far may also be extended to more general equations such as
$$
\left\{
\begin{alignedat}{4}
&-\dv \,A(x,\nabla u)+\beta(x,u)=-\,\dv \,(|f|^{p-2}f)+|g|^{r-2}g \ &&\mbox{ in } \Omega,\\
&u= 0 \ &&\mbox{ on } \ \partial \Omega,
\end{alignedat}
\right.
$$
where $1<q<p<r<\infty$, $1<s<r$, $f\in L^q_{\loc}(\overline{\Omega};\mathbb{R}^{d\times N})$, $g\in L^s_{\loc}(\overline{\Omega};\mathbb{R}^{N})$ and $A:\Omega\times\mathbb{R}^{d\times N}\to \mathbb{R}^{d\times N}$ is given as before. Moreover, $\beta:\Omega\times\mathbb{R}^N\to\mathbb{R}^N$ is also supposed to be a Carath\'eodory function satisfying the following assumptions: There exist constants $C_3,C_4>0$ and nonnegative functions $\beta_3\in L^1_{\loc}(\overline{\Omega})$, $\beta_4\in L^{r'}_{\loc}(\overline{\Omega})$ such that
\begin{gather*}
 \beta(x,z)\cdot z\geq C_3|z|^{r}-\beta_3(x),\\
 |\beta(x,z)|\leq C_4|z|^{r-1}+\beta_4(x),\\
 (\beta(x,z_1)-\beta(x,z_2))\cdot(z_1-z_2)\geq 0
\end{gather*}
for $z,z_1,z_2\inr^{N}$ and a.e.~$x\in\Omega$.
\end{remark}

\section*{Acknowledgments}
The authors wish to thank anonymous referees for their valuable comments. G.A.~is supported by JSPS KAKENHI Grant Numbers JP24H00184, JP21KK0044, JP21K18581, JP20H01812 and JP20H00117. This work was also supported by the Research Institute for Mathematical Sciences, an International Joint Usage/Research Center located in Kyoto University.

\bibliographystyle{plain}
\bibliography{bibliography}

\begin{thebibliography}{10}

\bibitem{AcMi05}
Emilio Acerbi and Giuseppe Mingione.
\newblock Gradient estimates for the {$p(x)$}-{L}aplacean system.
\newblock {\em J. Reine Angew. Math.}, 584:117--148, 2005.

\bibitem{B}
Viorel Barbu.
\newblock {\em Nonlinear semigroups and differential equations in {B}anach
  spaces}.
\newblock Editura Academiei Republicii Socialiste Rom\^{a}nia, Bucharest;
  Noordhoff International Publishing, Leiden, 1976.
\newblock Translated from the Romanian.

\bibitem{BGV93}
Lucio Boccardo, Thierry Gallou\"{e}t, and Juan~Luis V\'{a}zquez.
\newblock Nonlinear elliptic equations in {${\bf R}^N$} without growth
  restrictions on the data.
\newblock {\em J. Differential Equations}, 105(2):334--363, 1993.

\bibitem{Bre84}
H.~Brezis.
\newblock Semilinear equations in {${\bf R}^N$} without condition at infinity.
\newblock {\em Appl. Math. Optim.}, 12(3):271--282, 1984.

\bibitem{Bre}
Haim Brezis.
\newblock {\em Functional analysis, {S}obolev spaces and partial differential
  equations}.
\newblock Universitext. Springer, New York, 2011.

\bibitem{BBS16}
M.~Bul\'{i}\v{c}ek, J.~Burczak, and S.~Schwarzacher.
\newblock A unified theory for some non-{N}ewtonian fluids under singular
  forcing.
\newblock {\em SIAM J. Math. Anal.}, 48(6):4241--4267, 2016.

\bibitem{BBS19}
Miroslav Bul\'i\v{c}ek, Jan Burczak, and Sebastian Schwarzacher.
\newblock Well posedness of nonlinear parabolic systems beyond duality.
\newblock {\em Ann. Inst. H. Poincar\'e{} C Anal. Non Lin\'eaire},
  36(5):1467--1500, 2019.

\bibitem{BBKOS}
Miroslav Bul\'{i}\v{c}ek, Sun-Sig Byun, Petr Kaplick\'{y}, Jehan Oh, and
  Sebastian Schwarzacher.
\newblock On global {$L^q$} estimates for systems with {$p$}-growth in rough
  domains.
\newblock {\em Calc. Var. Partial Differential Equations}, 58(6):Paper No. 185,
  27, 2019.

\bibitem{BDS}
Miroslav Bul\'{\i}\v{c}ek, Lars Diening, and Sebastian Schwarzacher.
\newblock Existence, uniqueness and optimal regularity results for very weak
  solutions to nonlinear elliptic systems.
\newblock {\em Anal. PDE}, 9(5):1115--1151, 2016.

\bibitem{BuSch16}
Miroslav Bul\'{\i}\v{c}ek and Sebastian Schwarzacher.
\newblock Existence of very weak solutions to elliptic systems of
  {$p$}-{L}aplacian type.
\newblock {\em Calc. Var. Partial Differential Equations}, 55(3):Art. 52, 14,
  2016.

\bibitem{Dien1}
L.~Diening and S.~Fr\"{o}schl.
\newblock Extensions in spaces with variable exponents---the half space.
\newblock In {\em Harmonic analysis and nonlinear partial differential
  equations}, volume B22 of {\em RIMS K\^{o}ky\^{u}roku Bessatsu}, pages
  71--92. Res. Inst. Math. Sci. (RIMS), Kyoto, 2010.

\bibitem{Evans}
Lawrence~C. Evans.
\newblock {\em Partial differential equations}, volume~19 of {\em Graduate
  Studies in Mathematics}.
\newblock American Mathematical Society, Providence, RI, second edition, 2010.

\bibitem{GiaGiu82}
Mariano Giaquinta and Enrico Giusti.
\newblock On the regularity of the minima of variational integrals.
\newblock {\em Acta Math.}, 148:31--46, 1982.

\bibitem{GT}
David Gilbarg and Neil~S. Trudinger.
\newblock {\em Elliptic partial differential equations of second order}.
\newblock Classics in Mathematics. Springer-Verlag, Berlin, 2001.
\newblock Reprint of the 1998 edition.

\bibitem{Iwa83}
Tadeusz Iwaniec.
\newblock Projections onto gradient fields and {$L\sp{p}$}-estimates for
  degenerated elliptic operators.
\newblock {\em Studia Math.}, 75(3):293--312, 1983.

\bibitem{Iwa92}
Tadeusz Iwaniec.
\newblock {$p$}-harmonic tensors and quasiregular mappings.
\newblock {\em Ann. of Math. (2)}, 136(3):589--624, 1992.

\bibitem{KrMi06}
Jan Kristensen and Giuseppe Mingione.
\newblock The singular set of minima of integral functionals.
\newblock {\em Arch. Ration. Mech. Anal.}, 180(3):331--398, 2006.

\bibitem{Lew93}
John~L. Lewis.
\newblock On very weak solutions of certain elliptic systems.
\newblock {\em Comm. Partial Differential Equations}, 18(9-10):1515--1537,
  1993.

\bibitem{MiSc21}
Claudiu M\^{i}ndril\u{a} and Sebastian Schwarzacher.
\newblock Existence of steady very weak solutions to {N}avier-{S}tokes
  equations with non-{N}ewtonian stress tensors.
\newblock {\em J. Differential Equations}, 279:10--45, 2021.

\bibitem{Min10}
Giuseppe Mingione.
\newblock Nonlinear aspects of {C}alder\'on-{Z}ygmund theory.
\newblock {\em Jahresber. Dtsch. Math.-Ver.}, 112(3):159--191, 2010.

\bibitem{Min11}
Giuseppe Mingione.
\newblock Sketches of nonlinear {C}alder\'{o}n-{Z}ygmund theory.
\newblock In {\em Nonlinear analysis, function spaces and applications. {V}ol.
  9}, pages 105--144. Acad. Sci. Czech Repub. Inst. Math., Prague, 2011.

\bibitem{Min17}
Giuseppe Mingione.
\newblock Short tales from nonlinear {C}alder\'{o}n-{Z}ygmund theory.
\newblock In {\em Nonlocal and nonlinear diffusions and interactions: new
  methods and directions}, volume 2186 of {\em Lecture Notes in Math.}, pages
  159--204. Springer, Cham, 2017.

\bibitem{Nec77}
Jindrich Ne\v{c}as.
\newblock Example of an irregular solution to a nonlinear elliptic system with
  analytic coefficients and conditions for regularity.
\newblock In {\em Theory of nonlinear operators ({P}roc. {F}ourth {I}nternat.
  {S}ummer {S}chool, {A}cad. {S}ci., {B}erlin, 1975)}, volume No. 1 of {\em
  Abhandlungen der Akademie der Wissenschaften der DDR, Abteilung Mathematik,
  Naturwissenschaften, Technik}, pages 197--206. Akademie-Verlag, Berlin, 1977.

\bibitem{21}
Alberto Torchinsky.
\newblock {\em Real-variable methods in harmonic analysis}, volume 123 of {\em
  Pure and Applied Mathematics}.
\newblock Academic Press, Inc., Orlando, FL, 1986.

\bibitem{22}
Bengt~Ove Turesson.
\newblock {\em Nonlinear potential theory and weighted {S}obolev spaces},
  volume 1736 of {\em Lecture Notes in Mathematics}.
\newblock Springer-Verlag, Berlin, 2000.

\end{thebibliography}

\end{document}